\pdfoutput=1
\documentclass[USenglish]{jnsao}

\input{TeX/Preamble.tex}

\newcommand{\TheCodeZenodoDOI}{10.5281/zenodo.11098283}

\title{A penalty barrier framework\texorpdfstring{\\}{ }for nonconvex constrained optimization}
\shorttitle{A penalty barrier framework for constrained optimization}
\author{%
	Alberto De~Marchi\thanks{\TheAddressUBM. \email{alberto.demarchi@unibw.de}, \textsc{orcid}: \orcidLink{0000-0002-3545-6898}}%
	\and
	Andreas Themelis\thanks{\TheAddressKUJ. \email{andreas.themelis@ees.kyushu-u.ac.jp}, \textsc{orcid}: \orcidLink{0000-0002-6044-0169}}%
}
\shortauthor{De~Marchi and Themelis}
\acknowledgements{%
	Work supported by the Japan Society for the Promotion of Science (JSPS) KAKENHI grants JP21K17710 and JP24K20737.
	The authors also gratefully acknowledge the constructive feedback provided by the anonymous reviewers, whose careful reading and thoughtful remarks have been instrumental in refining this manuscript.%
}

\manuscriptsubmitted{2024-10-17}
\manuscriptaccepted{2025-08-14}
\manuscriptvolume{5}
\manuscriptnumber{14585}
\manuscriptyear{2025}
\manuscriptdoi{10.46298/jnsao-2025-14585}

\manuscriptlicense{CC-BY 4.0}

\showgrayfalse
\renewcommand{\includetikz}[2][]{\includegraphics[#1]{Pics/Tikz/#2.pdf}}
\begin{document}

	\maketitle

	\begin{abstract}
		We consider minimization problems with structured objective function and smooth constraints,{}  and present a flexible{}  framework that combines the beneficial regularization effects of (exact) penalty and interior-point methods.
In the fully nonconvex setting, a pure barrier approach requires careful steps when approaching the infeasible set, thus hindering convergence.
We show how a tight integration with a penalty scheme mitigates this issue and enables the construction of subproblems whose domain is independent of the explicit constraints.
This decoupling allows us to leverage efficient solvers designed for unconstrained or suitably structured optimization tasks.
The key behind all this is a marginalization step:
closely related to a conjugacy operation, this step effectively merges (exact) penalty and barrier into a smooth, full domain functional object.
When the penalty exactness takes effect, the generated subproblems do not suffer the ill-conditioning typical of barrier methods,
nor do they exhibit the nonsmoothness of exact penalty terms.
We provide a theoretical characterization of the algorithm and its asymptotic properties, deriving convergence results for fully nonconvex problems.
Stronger conclusions are available for the convex setting, where optimality can be guaranteed.
Illustrative examples and numerical simulations demonstrate the wide range of problems our theory and algorithm are able to cover.

		\bigskip
		
		\noindent
		{\color{structure}Keywords}~~%
		Nonsmooth nonconvex optimization,
		exact penalty methods,
		interior point methods,
		proximal algorithms
		
		\noindent
		{\color{structure}AMS subject classifications}~~%
		\amsmscLink{49J52},
		\amsmscLink{49J53},
		\amsmscLink{65K05},
		\amsmscLink{90C06},
		\amsmscLink{90C30}%
	\end{abstract}

	\section{Introduction}
		We are interested in developing numerical methods for constrained optimization problems of the form
\[\tag{P}\label{eq:P}
	\minimize_{\x\in\R^n}~\cost(\x)
\quad
	\stt{}~
	\c(\x)\leq\zeros,~
	\ceq(\x)=\zeros,
\]
where \(\x\) is the decision variable and \(\cost\), \(\c\) and \(\ceq\) are problem functions.
(Throughout, we stick to the convention of bold-facing vector variables and vector-valued functions, so that \(\zeros\) indicates the zero vector of suitable size and similarly \(\ones\) is the vector with all entries equal to one.)
Henceforth we consider \eqref{eq:P} under the following standing assumptions.

\begin{mybox}
	\begin{assumption}\label{ass:basic}%
		The following hold in problem \eqref{eq:P}:
		\begin{enumeratass}
		\item \label{ass:cost}%
			\(\func{\cost}{\R^n}{\Rinf}\) is proper and lower semicontinuous (lsc).
		\item \label{ass:c}%
			\(\func{\c}{\R^n}{\R^m}\) and \(\func{\ceq}{\R^n}{\R^\meq}\) are continuously differentiable.
		\item \label{ass:feas}%
			The problem is well posed:
			\(\cost*\coloneqq\inf_{\set{\x}[\c(\x)\leq\zeros,~\ceq(\x)=\zeros]}\cost(\x)\) is finite.
		\end{enumeratass}
	\end{assumption}
\end{mybox}

Notice that no differentiability requirements are imposed on the cost \(\cost\), nor convexity on any term in the formulation (hence the connotation of \emph{fully} nonconvex).
This modeling flexibility allows one to include simple constraints directly in \(\cost\), forcing all generated iterates to honor them, as an alternative to explicit constraints in the format $\c(\x)\leq\zeros$, $\ceq(\x)=\zeros$, which may be violated along the iterates.
	We remark that our framework allows for (and is robust to) equality constraints \(\ceq(\x)=\zeros\) to be expressed as two-sided inequalities \(\ceq(\x)\leq\zeros\) and \(-\ceq(\x)\leq\zeros\), despite the lack of constraint qualifications and in contrast to purely interior-point schemes, though a dedicated treatment of equalities yields an advantage in terms of algorithmic performance.

The primary objective of this paper is to devise an abstract algorithmic framework in the generality of this setting.
The methodology requires an oracle for solving, up to approximate local optimality, minimization instances of the sum of \(\cost\) with a differentiable term.
	In practice, some structure is required in the cost function \(\cost\) to efficiently address these subproblems.
	The general setting of \eqref{eq:P} under \cref{ass:basic} is considered without significantly weakening the convergence guarantees with respect to, say, assuming \(q\) smooth.
	Nevertheless, some stronger results are established under additional assumptions, such as locally Lipschitz continuity of the cost \(\cost\) or convexity of \eqref{eq:P}.
In our numerical experiments we will invoke off-the-shelf routines based on proximal gradient iterations, thereby restricting our attention to problem instances in which \(\cost\) is structured as \(\cost=f+g\) for a differentiable function \(f\) and a function \(g\) that enjoys an easily computable proximal map.
Most nonsmooth functions widely used in practice comply with all these requirements.
For instance, \(g\) can include indicators of any nonempty and closed set, and thus enforce arbitrary closed constraints that are easy to project onto.
	This modeling flexibility extends beyond the handling of constraints.
	While nonsmooth functions commonly encountered in practice can often be reformulated into smooth equivalents using slack variables and additional constraints, such reformulations typically come at the cost of introducing auxiliary variables and complicating the problem structure, penalizing algorithmic efficiency.
	By allowing nonsmooth terms to appear directly in the objective, our framework eliminates the need for such artificial constructs.

A particularly illustrative example is the so-called \(L^0\)-(pseudo)norm penalty (number of nonzero entries) \(\|\x\|_0\) for \(\x\in\R^n\).
As shown in \cite[Lem. 3.1]{bi2017multistage}, this function can be represented as the \emph{linear program}
\begin{equation}\label{eq:l0_as_linprog}
	\|\x\|_0
=
	\min_{\vec{u}\in\R^n}{}
	\|\vec{u}\|_1
\quad
	\stt{}~
	-\ones\leq\vec{u}\leq\ones,~
	\innprod{\vec{u}}{\x}
	=
	\|\x\|_1,
\end{equation}
(more generally, matrix rank can also be cast in a similar fashion).
	In turn, nonsmoothness of each of the \(L^1\)-norms can be resolved via the introduction of \(n\) slack variables and \(2n\) inequality constraints, leading to a substantial increase in both problem size and constraint count.
	In contrast, our approach accommodates nonsmooth terms such as the \(L^0\)-norm directly in the objective, avoiding any such inflation.
	The computational advantages of this modeling flexibility are also evident from the numerical experiments presented in \cref{sec:Numerics:MatrixCompletion}.

\paragraph*{Motivations and related work}
The class of problems \eqref{eq:P} with structured cost \(\cost\) has been recently studied in \cite{chouzenoux2020proximal} and \cite{demarchi2024interior}, respectively, for the fully convex and nonconvex setting, developing methods that bear strong convergence guarantees under some restrictive assumptions.
Above all, building on a pure barrier approach, these methods demand a feasible set with nonempty interior, thus excluding problems with equality constraints.
Although restricted to simple bounds, a similar interior-point technique is investigated in \cite{leconte2024interior} and manifests analogous pros and cons.
In contrast to these works, we intend to address equality constraints as well.
An augmented Lagrangian scheme for constrained structured problems was developed in \cite{demarchi2023constrained}, which also allows the specification of constraints in a function-in-set format.

Constrained structured programs \eqref{eq:P} are also closely related to the template of structured \emph{composite} optimization
\[
	\minimize_{\x\in\R^n}{}~
	\cost(\x) + h(\c(\x))
\]
with \(\func{h}{\R^m}{\Rinf}\).
By introducing additional variables, composite problems can be rewritten in (equality) constrained form recovering the class of problems \eqref{eq:P},
with a one-to-one relationship between (local and global) solutions and stationary points \cite[Lem. 3.1]{demarchi2023constrained}.
The recent literature on
structured composite optimization includes
\cite{rockafellar2022convergence}, only for convex \(h\),
and \cite{hallak2023adaptive,demarchi2024implicit} for fully nonconvex problems,
and concentrates almost exclusively on the augmented Lagrangian framework.
Relying essentially on a penalty approach, in contrast to a barrier, the algorithmic characterization in \cite{demarchi2023constrained} involved weaker assumptions and yet retrieved standard convergence results in constrained nonconvex optimization.
However, the dependency on dual estimates makes methods of this family sensitive to the initialization of Lagrange multipliers.
Moreover, they require some safeguards to ensure convergence from arbitrary starting points \cite{birgin2014practical,conn1991globally}.
In contrast, thanks to their `primal' nature and inherent regularizing effect,
penalty-barrier techniques can conveniently cope with degenerate problems.

The idea of adopting and merging penalty and barrier approaches, in a variety of possible flavors and combinations, is certainly not new, tracing back at least to \cite{fiacco1964sequential}.
Among several recent concretizations of this avenue, we refer to Curtis' work \cite{curtis2012penalty} for a comprehensive discussion and further references.
Our motivation for developing this technique for constrained structured problems
stems from insights gained during the design of the interior point scheme \ipprox{} \cite{demarchi2024interior}.
The key observation therein is that, with a pure barrier approach, the arising subproblems have a smooth term \emph{without} full domain.
This nonstandard situation, together with a nonconvex and possibly extended-real-valued cost \(\cost\) and nonlinear constraints \(\c(\x)\leq\zeros\), significantly restricts the range of subsolvers that can be employed.
As a result, one cannot fully exploit more efficient optimization routines that would otherwise be suitable in an unconstrained or more structured setting.

In the broad setting of \eqref{eq:P} under \cref{ass:basic},
a blind application of penalty-barrier strategies in the spirit of \cite{curtis2012penalty} would bear no advantages,
since the inconvenience in \ipprox{} of a restricted domain would persist,
hindering again the practical performance.
In this paper we propose and investigate in detail a simple technique to overcome this limitation.
The crucial step consists in the \emph{marginalization} of auxiliary variables:
after applying some penalty and barrier modifications,
the auxiliary variables are optimized \emph{pointwise}, for any given decision variable \(\x\).%
\footnote{%
	This approach can be interpreted as a drastic version of the so-called \emph{magical steps} \cite{conn2000trust,birgin2014practical}, or \emph{slack reset} in \cite{curtis2012penalty},
	and was inspired by the \emph{proximal} approaches in \cite{dhingra2019proximal,demarchi2024implicit}.%
}
Before proceeding with the technical content,
we emphasize that the marginalization step not only reduces the subproblems' size (recovering that of the original decision variable \(\x\) only),
but it also---and especially---results in a smooth penalty term for the subproblems that has always \emph{full} domain.
The emergence of this penalty-barrier envelope enables the adoption of generic (efficient) subsolvers, as well as tailored routines that exploit the problem's original structure.
This claim will be substantiated in \cref{sec:IP}, where we show that properties such as convexity and Lipschitz differentiability, whenever present, are preserved in the transformed problems.

	\section{Preliminaries}
		In this section we comment on useful notation and preliminary results before discussing optimality notions to characterize solutions of \eqref{eq:P}.

		\subsection{Notation and known facts}
			With \(\R\) and \(\Rinf \coloneqq \R \cup \set{\pm\infty}\) we denote the real and extended-real line, respectively, and with \(\R_+\coloneqq[0,\infty)\) and \(\R_-\coloneqq(-\infty,0]\) the set of nonnegative and nonpositive real numbers, respectively.
The positive and negative parts of a number \(r\in\R\) are respectively denoted as \([r]_+\coloneqq\max\set{0,r}\) and \([r]_-\coloneqq\max\set{0,-r}\), so that \(r=[r]_+-[r]_-\) and \(|r|=[r]_++[r]_-\).
We stick to the convention of bold-facing vector variables and vector-valued functions, and use \(\zeros\) to denote the zero vector of suitable size and similarly \(\ones\) for the vector with all entries equal to one.
When applying unary operators to a vector \(\vec r\), such as \(|\vec r|\) or \([\vec r]_+\), the operation is meant elementwise.

The notation \(\ffunc{T}{\R^n}{\R^m}\) indicates a set-valued mapping \(T\) that maps any \(\x\in\R^n\) to a (possibly empty) subset \(T(\x)\) of \(\R^m\).
Its \DEF{(effective) domain} and \DEF{graph} are the sets \(\dom T\coloneqq\set{\x\in\R^n}[T(\x)\neq\emptyset]\) and \(\graph T\coloneqq\set{(\x,\y)\in\R^n\times\R^m}[\y\in T(\x)]\).
	\(T\) is said to be \DEF{outer semicontinuous} (osc) if its graph is a closed subset of \(\R^n\times\R^m\).
Algebraic operations with or among set-valued mappings are meant in a componentwise sense; for instance, the sum of \(\ffunc{T_1,T_2}{\R^n}{\R^m}\) is defined as \((T_1+T_2)(\x)\coloneqq\set{\y^1+\y^2}[(\y^1,\y^2)\in T_1(\x)\times T_2(\x)]\) for all \(\x\in\R^n\).

The \DEF{distance} from a nonempty set \(E\subseteq\R^n\) \(\func{\dist_E}{\R^n}{[0,\infty)}\) is
\(
	\dist_E(\x)\coloneqq\inf_{\y\in E}\|\y-\x\|
\).
With \(\func{\indicator_E}{\R^n}{\Rinf}\) we denote the \DEF{indicator function} of \(E\), namely such that \(\indicator_E(\x)=0\) if \(\x\in E\) and \(\infty\) otherwise.
For an extended-real-valued function \(\func{h}{\R^n}{\Rinf}\), the \DEF{(effective) domain}, \DEF{graph}, and \DEF{epigraph} are given by \(\dom h \coloneqq \set{\x\in\R^n}[h(\x) < \infty]\), \(\graph h\coloneqq\set{(\x,h(\x))}[\x\in\dom h]\), and \(\epi h\coloneqq\set{(\x,\alpha)\in\R^n\times\R}[\alpha\geq h(\x)]\).
We say that \(h\) is \DEF{proper} if \(\dom h \neq \emptyset\) and \(h>-\infty\), and \DEF{lower semicontinuous} (lsc) if \(h(\xbar) \leq \liminf_{\x\to\xbar} h(\x)\) for all \(\xbar \in \R^n\) or, equivalently, if \(\epi h\) is a closed subset of \(\R^{n+1}\).
Following \cite[Def. 8.3]{rockafellar1998variational}, we denote by \(\ffunc{\hat\partial h}{\R^n}{\R^n}\) the \emph{regular subdifferential} of \(h\), where
\begin{equation*}
	\vec{v} \in \hat\partial h(\xbar)
\quad\defeq[\Leftrightarrow]\quad
	\liminf_{\limsubstack{\x&\to&\xbar\\ \x&\neq&\xbar}} \frac{h(\x) - h(\xbar) - \innprod{\vec{v}}{\x-\xbar}}{\|\x-\xbar\|} \geq 0 .
\end{equation*}
The (\DEF{limiting}, or \DEF{Mordukhovich}) \emph{subdifferential} of \(h\) is \(\ffunc{\partial h}{\R^n}{\R^n}\), where \(\bar{\vec{v}} \in \partial h(\xbar)\) if and only if \(\xbar\in\dom h\) and there exists a sequence \(\seq{\x^k,\vec{v}^k}\) in \(\graph\hat\partial h\) such that \((\x^k,\vec{v}^k,h(\x^k))\to(\xbar,\bar{\vec{v}},h(\xbar))\).
In particular, \(\hat\partial h(\x)\subseteq\partial h(\x)\) holds at any \(\x\in\R^n\); moreover, \(\vec{0}\in\hat\partial h(\x)\) is a necessary condition for local minimality of \(h\) at \(\x\) \cite[Thm. 10.1]{rockafellar1998variational}.
The subdifferential of \(h\) at \(\xbar\) satisfies \(\partial(h+h_0)(\xbar) = \partial h(\xbar) + \nabla h_0(\xbar)\) for any \(\func{h_0}{\R^n}{\Rinf}\) continuously differentiable around \(\xbar\) \cite[Ex. 8.8]{rockafellar1998variational}.
If \(h\) is convex, then \(\hat\partial h=\partial h\) coincide with the \DEF{convex subdifferential}
\[
	\R^n\ni\xbar
\mapsto
	\set{\vec{v}\in\R^n}[{h(\x) - h(\xbar) - \innprod{\vec{v}}{\x-\xbar}\geq0\ \forall\x\in\R^n}].
\]
For a convex set \(C\subseteq\R^m\) and a point \(\x\in C\) one has that \(\partial\indicator_C(\x)=\ncone_C(\x)\), where
\[
	\ncone_C(\x)
\coloneqq
	\set{\vec{v}\in\R^n}[\innprod{\vec{v}}{\x'-\x}\leq 0 \ \forall \x'\in C]
\]
denotes the \DEF{normal cone} of \(C\) at \(\x\), while \(\ncone_C(\x)=\emptyset\) for \(\x\notin C\).

We use the symbol \(\func{\jac\vec{F}}{\R^n}{\R^{m\times n}}\) to indicate the Jacobian of a differentiable mapping \(\func{\vec{F}}{\R^n}{\R^m}\), namely \(\jac\vec{F}(\xbar)_{i,j}=\dep{\vec{F}_i}{\x_j}(\xbar)\) for all \(\xbar\in\R^m\).
For a real-valued function \(h\), we instead use the gradient notation \(\nabla h\coloneqq\trans{\jac h}\) to indicate the column vector of its partial derivatives.
Finally, we remind that the \DEF{convex conjugate} of a proper lsc convex function \(\func{b}{\R}{\Rinf}\) is the proper lsc convex function \(\func{\conj b}{\R}{\Rinf}\) defined as
\(
	\conj b(\tau)\coloneqq\sup_{t\in\R}\set{\tau t-b(t)}
\),
and that one then has \(\tau\in\partial b(t)\) if and only if \(t\in\partial\conj b(\tau)\).

		\subsection{Stationarity concepts}\label{sec:stationarity}%
			This subsection summarizes well-known standard local optimality measures which were adopted in the proximal interior point framework of \cite{demarchi2024interior}, and which will be further developed in the following \cref{sec:Subproblems} into conditions tailored to the setting of this paper.
The interested reader is referred to \cite[\S2]{demarchi2024interior} for a verbose introduction and to \cite[\S3]{birgin2014practical} for a detailed treatise.
We start with the usual notion of (approximate) stationarity for general minimization problems of an extended-real-valued function.

\begin{definition}[stationarity]\label{defin:stationary}%
	Relative to the problem
	\(
		\minimize_{\x\in\R^n}\varphi(\x)
	\)
	for a function \(\func{\varphi}{\R^n}{\Rinf}\),
	a point \(\xbar\in\R^n\) is called
	\begin{enumerate}
	\item
		\DEF{stationary} if it satisfies \(\zeros\in\partial\varphi(\xbar)\);
	\item
		\DEF{\(\varepsilon\)-stationary} (with \(\varepsilon>0\)) if it satisfies \(\dist_{\partial\varphi(\xbar)}(\zeros)\leq\varepsilon\).
	\end{enumerate}
\end{definition}

A standard optimality notion that reflects the constrained structure of \eqref{eq:P} is given by the Karush-Kuhn-Tucker (KKT) conditions.

\begin{definition}[\protect\KKT* optimality]\label{defin:KKT}%
	Relative to problem \eqref{eq:P}, we say that \(\xbar\in\R^n\) is \DEF{KKT-optimal} if there exist \(\ybar\in\R^m\) and \(\yeqbar\in\R^\meq\) such that
	\[\tag{KKT}\label{KKT}
		\begin{cases}
			-\trans{\jac \c(\xbar)}\ybar
			-\trans{\jac \ceq(\xbar)}\yeqbar
			\in
			\partial\cost(\xbar)
		\\
			\c(\xbar)\leq\zeros
			~\text{and}~
			\ceq(\xbar)=\zeros
		\\
			\ybar\geq\zeros
		\\
			\ybar_i\c_i(\xbar)=0\quad i=1,\dots,m.
		\end{cases}
	\]
	In such case, we say that \((\xbar,\ybar,\yeqbar)\in\R^n\times\R^m\times\R^\meq\) is a \DEF{KKT-optimal triplet} for \eqref{eq:P}.
\end{definition}

Even for convex problems, unless suitable constraint and epigraphical qualifications are met, local minimizers may fail to be \KKT-optimal.
Necessary conditions in the generality of problem \eqref{eq:P} are provided by the following asymptotic counterpart.\footnote{%
	This definition is inspired by \cite[Def. 3.1]{birgin2014practical}, where the `A' in \AKKT* is short for `approximate'.
	We however find `asymptotic' more fit to emphasize its dependency on sequences, and reserve the `approximate' label to characterize points satisfying \KKT* optimality up to some tolerance as in \cref{defin:eKKT}.
}

\begin{definition}[\protect\AKKT* optimality]\label{defin:AKKT}%
	Relative to problem \eqref{eq:P}, we say that \(\xbar\in\R^n\) is \DEF{asymptotically KKT-optimal} if \(\xbar\in\dom\cost\) and there exist sequences \(\seq{\x^k}\to\xbar\), \(\seq{\y^k}\subset\R^m\) and \(\seq{\yeq^k}\subset\R^\meq\) such that
	\[\tag{A-KKT}\label{AKKT}
		\renewcommand{\arraystretch}{1.2}
		\begin{cases}
			\dist_{\partial q(\x^k)}\bigl(
				-\trans{\jac \c(\x^k)}\y^k
				-\trans{\jac \ceq(\x^k)}\yeq^k
			\bigr)
			\to
			0
		\\{}
			[\c(\x^k)]_+\to \zeros
			~\text{and}~
			\ceq(\x^k)\to \zeros
		\\
			\y^k\geq \zeros
		\\
			\y_i^k\c_i(\xbar)=0\quad i=1,\dots,m.
		\end{cases}
	\]
\end{definition}

The requirement \(\xbar\in\dom\cost\), while superfluous in the original \cite[Def. 3.1]{birgin2014practical},
is a necessary technicality to cope with possible nonclosedness of \(\dom\cost\) in the generality of \cref{ass:basic}.
Taking the unconstrained minimization of \(\cost(x)=\frac{1}{|x|}+\sin\frac{1}{x}\) as an example, this requirements prevents \(\bar x=0\notin\dom\cost\) to be considered A-KKT-optimal despite the fact that \(x^k=\frac{1}{(2k+1)\pi}\to\bar x\) constitutes a valid sequence in the definition (having \(\dist_{\partial\cost(x^k)}(0)=0\) for all \(k\)).

\begin{proposition}[{\cite[Thm. 3.1]{birgin2014practical}, \cite[Prop. 2.5]{demarchi2023constrained}}]
	Any local minimizer for \eqref{eq:P} is \AKKT-optimal.
\end{proposition}

For the sake of designing suitable algorithmic stopping criteria, we also define an approximate variant which provides a further weaker notion of optimality.

\begin{definition}[\eKKT* optimality]\label{defin:eKKT}%
	Relative to problem \eqref{eq:P}, for \(\vec\epsilon=(\epsilon_{\rm p},\epsilon_{\rm d})>(0,0)\) we say that \(\xbar\) is an \DEF{(approximate) \(\vec\epsilon\)-KKT point} if there exist \(\ybar\in\R^m\) and \(\yeqbar\in\R^\meq\) such that
	\[\tag{\eKKT*}\label{eKKT}
		\renewcommand{\arraystretch}{1.2}
		\begin{cases}
			\dist_{\partial q(\xbar)}\bigl(
				-\trans{\jac \c(\xbar)}\ybar
				-\trans{\jac \ceq(\xbar)}\yeqbar
				\bigr)\leq\epsilon_{\rm d}
		\\
			\|[\c(\xbar)]_+\|_\infty\leq\epsilon_{\rm p}
			~\text{and}~
			\|\ceq(\xbar)\|_\infty\leq\epsilon_{\rm p}
		\\
			\ybar\geq \zeros
		\\
			\min\set{\ybar_i,[\c_i(\xbar)]_-}\leq\epsilon_{\rm p}\quad i=1,\dots,m.
		\end{cases}
	\]
\end{definition}

It is handy to name points satisfying (approximate) feasibility as in \cref{defin:eKKT} and \cref{defin:KKT} in order to soften symbolic clutter in the sequel.

\begin{definition}[\(\epsilon\)-feasibility]
	Given \(\epsilon\geq0\), a point \(\xbar\in\R^n\) is said to be \DEF{\(\epsilon\)-feasible} if \(\|[\c(\xbar)]_+\|_\infty\leq\epsilon\) and \(\|\ceq(\xbar)\|_\infty\leq\epsilon\).
	When \(\epsilon=0\), we simply say that \(\xbar\) is \DEF{feasible}.
\end{definition}

As discussed in the commentary after \cite[Thm. 3.1]{birgin2014practical}, \AKKT{} optimality of \(\xbar\in\dom\cost\) is tantamount to the existence of a sequence \(\x^k\to\xbar\) of \eKKT[\vec{\epsilon}^k] points for some \(\vec{\epsilon}^k\to(0,0)\).
More generally, any \KKT{} point is both \AKKT{} and \eKKT{} for any \(\vec{\epsilon}\geq(0,0)\).
We conclude by listing the observations in \cite[Lem. 8 and Rem. 9]{demarchi2024interior} that will be useful in the sequel.

\begin{remark}\label{thm:yk}%
	Relative to the conditions \AKKT{} in \cref{defin:AKKT}:
	\begin{enumerate}
	\item \label{thm:ykck}%
		Up to possibly perturbing the sequence of multipliers, the complementarity slackness \(\y_i^k\c_i(\xbar)=0\) can be equivalently expressed as \(\y_i^k\c_i(\x^k)\to0\).
	\item \label{thm:ykbounded}%
		Suppose that \(\partial\cost\) is osc on \(\dom\cost\) (as is the case when \(\cost\) is continuous relative to its domain).
		If the sequence \(\seq{\y^k,\yeq^k}\) contains a bounded subsequence, then \(\xbar\) is a \KKT-optimal point, not merely asymptotically.
	\qedhere
	\end{enumerate}
\end{remark}

We note that \cite{demarchi2024interior} imposes a standing assumption that \(\cost\) be continuous at every point \(\xbar\in\dom\cost\).
This is used to ensure that \(\vec{v}[\bar]\in\partial\cost(\xbar)\) whenever \((\x^k,\vec{v}^k)\to(\xbar,\vec{v}[\bar])\) with \((\x^k,\vec{v}^k)\in\graph\partial\cost\), that is, that the condition \(\cost(\x^k)\to\cost(\xbar)\) is superfluous in the definition of limiting subdifferential.
In \cref{thm:ykbounded} we have relaxed this requirement by directly assuming this limiting property, that is, that \(\partial\cost\) is osc on \(\dom\cost\).
Many functions of practical interest that are not continuous, such as the \(L^0\)-norm, still comply with this requirement.

	\section{Subproblems generation}\label{sec:Subproblems}%
		In this section we operate a two-step modification of problem \eqref{eq:P}, whose conceptual roadmap is as follows.
We begin with a relaxed reformulation \eqref{eq:Pa} in which violation of the constraints \(\c(\x)\leq\zeros\) and \(\ceq(\x)=\zeros\) is penalized with an \(L^1\)-norm in the cost function.
An equivalent reformulation \eqref{eq:Qa} with slack variables \(\s\in\R^m\) and \(\sseq\in\R^\meq\) simplifies this formulation by promoting separability.
Next, a new problem \eqref{eq:Qam} is created by adding a barrier term to enforce strict satisfaction of the inequality constraints in the \(L^1\)-penalized reformulation \eqref{eq:Qa}.
The pointwise minimization with respect to the slack variables \(\s\) and \(\sseq\) can be carried out explicitly, with negligible computational overhead, resulting in a new problem \eqref{eq:Pam} in which the original constraints \(\c(\x)\leq\zeros\) and \(\ceq(\x)=\zeros\) are softened with a smooth penalty.
Increasing the \(L^1\)-penalty and decreasing the barrier coefficients gives rise to a homotopic transition between smooth reformulations \eqref{eq:Pam} and the original nonsmooth problem \eqref{eq:P}.

Compared to envelope-type smoothings such as in \cite{stella2017forward} that preserve one-to-one correspondence of minimizers for any parameter, the smoothened subproblems here are equivalent to the original \eqref{eq:P} only in the limit.
In this sense, our method is more closely related to the approach in \cite{simoes2021lasry}, but without the practical restrictions.
Unlike the latter, which is tailored to problems where nonsmooth terms admit an easily computable `double envelope', our technique applies more broadly with minimal limitations.

We adopt the convention of using the label `P' in problems \eqref{eq:Pa} and \eqref{eq:Pam} that share the minimization variable \(\x\) with that in the original problem \eqref{eq:P}.
Label `Q' is instead used in problems \eqref{eq:Qa} and \eqref{eq:Qam} which come with additional slack variables \(\s\) and \(\sseq\).
Each `P-problem' amounts to the corresponding `Q-problem' after marginal minimization with respect to the slack variables.

		\subsection{\texorpdfstring{\(\bm{L^1}\)}{L1}-penalization}\label{sec:L1}%
			Given \(\alpha>0\), we consider the following \(L^1\) relaxation of \eqref{eq:P}:
\[
	\tag{\({\rm P}\!_\alpha\)}\label{eq:Pa}
	\minimize_{\x\in\R^n}~
		\cost(\x)
		+
		\alpha\|[\c(\x)]_+\|_1
		+
		\alpha\|\ceq(\x)\|_1.
\]
By introducing slack variables \(\s\in\R^m\) and \(\sseq\in\R^\meq\), \eqref{eq:Pa} can equivalently be cast as
\[
	\tag{\({\rm Q}_\alpha\)}\label{eq:Qa}
	\begin{array}[t]{>{\displaystyle}r @{\ } >{\displaystyle}l}
		\minimize_{\substack{
				\x\in\R^n,\s\in\R^m\\
				\sseq\in\R^\meq
		}}~
		&
		\cost(\x)
		+
		\alpha\innprod{\ones}{\s}
		+
		\indicator_{\R_+^m}(\s)
		+
		\alpha\innprod{\ones}{\sseq}
	\\
		\stt{}~& \c(\x) \leq \s
		~~\text{and}~~
		{-\sseq}\leq\ceq(\x)\leq\sseq
		,
	\end{array}
\]
as one can easily verify that
\begin{align*}
	[\c(\x)]_+
={} &
	\argmin_{\s\in\R^m}\set{\alpha\innprod{\ones}{\s}+\indicator_{\R_+^m}(\s)}[\c(\x) \leq \s]
\shortintertext{and}
	|\ceq(\x)|
={} &
	\argmin_{\sseq\in\R^\meq}\set{\alpha\innprod{\ones}{\sseq}}[{-\sseq}\leq\ceq(\x) \leq \sseq]
\end{align*}
hold for any \(\x\in\R^n\) and \(\alpha>0\).
In other words, \eqref{eq:Pa} amounts to \eqref{eq:Qa} after a marginal minimization with respect to the slack variables \(\s\) and \(\sseq\).
The KKT conditions associated to \eqref{eq:Qa} are of particular interest to us.
As it can be deduced from the following lemma, they correspond to the stationarity condition \(\zeros\in\partial\cost+\alpha\partial\bigl[\|[\c({}\cdot{})]_+\|_1\bigr]+\alpha\partial\bigl[\|[\ceq({}\cdot{})\|_1\bigr]\) for \eqref{eq:Pa}.
The result is textbook, but its proof is nevertheless detailed in \cref{proof:thm:KKTa} for completeness.

\begin{lemma}\label{thm:KKTa}%
	Let \cref{ass:basic} hold.
	Then, a point \((\x,\s,\sseq)\in\R^n\times\R^m\times\R^{\meq}\) is KKT-optimal for \eqref{eq:Qa} if and only if \(\s=[\c(\x)]_+\), \(\sseq=|\ceq(\x)|\), and there exist \(\y\in\R^m\) and \(\yeq\in\R^{\meq}\) such that%
	\[\tag{KKT\(_\alpha\)}\label{KKTa}
		\begin{cases}
			-\trans{\jac\c(\x)}\y-\trans{\jac\ceq(\x)}\yeq
			\in
			\partial\cost(\x) &
		\\
			\zeros\leq\y\leq\alpha\ones
		\\
			|\yeq|\leq\alpha\ones
		\\
			 \y_i[\c_i(\x)]_-=0=(\alpha-\y_i)[\c_i(\x)]_+ ,~~~i=1,\dots,m%
		\\
			 (\alpha-\yeq_j)[\ceq_j(\x)]_+=0=(\alpha+\yeq_j)[\ceq_j(\x)]_- ,~~~j=1,\dots,\meq.%
		\end{cases}
	\]
\end{lemma}
\begin{proof}
	See \cref{proof:thm:KKTa}.
\end{proof}

\Cref{thm:KKTa} suggests the following relaxed optimality notion for problem \eqref{eq:P}, which in light of the connection with KKT-optimality for \eqref{eq:Qa} we shall refer to as \KKTa-optimality.

\begin{definition}[\protect\KKTa* optimality]%
	Given \(\alpha>0\), we say that a point \(\xbar^\alpha\in\R^n\) is \DEF{\KKTa*-optimal} for \eqref{eq:P} if there exist \(\ybar^\alpha\in\R^m\) and \(\yeqbar^\alpha\in\R^{\meq}\) such that \((\xbar^\alpha,\ybar^\alpha,\yeqbar^\alpha)\) satisfy \eqref{KKTa}, and call \((\xbar^\alpha,\ybar^\alpha,\yeqbar^\alpha)\in\R^n\times\R^m\times\R^\meq\) a \DEF{\KKTa*-optimal triplet} for \eqref{eq:P}.
\end{definition}

Similarly to what done in \eqref{eKKT} with respect to \eqref{KKT}, we may introduce an approximate \KKTa-optimality condition in which stationarity and complementarity slackness are satisfied up to some tolerance parameters.
When said tolerance is zero, the nonapproximate \KKTa{} notion is recovered.

\begin{definition}[\protect\eKKTa* optimality]%
	Given \(\alpha>0\) and \(\vec{\epsilon}=(\epsilon_{\rm p},\epsilon_{\rm d})\geq(0,0)\), we say that a point \(\xbar^\alpha\in\R^n\) is \DEF{\eKKTa*-optimal} for \eqref{eq:P} if there exist \(\ybar^\alpha\in\R^m\) and \(\yeqbar^\alpha\in\R^\meq\)
	such that
	\[\tag{\(\vec{\epsilon}\)-KKT\(_\alpha\)}\label{eKKTa}
		\renewcommand{\arraystretch}{1.2}
		\begin{cases}
			\dist_{\partial\cost(\xbar^\alpha)}\bigl(
				-\trans{\jac \c(\xbar^\alpha)}\ybar^\alpha
				-\trans{\jac \ceq(\xbar^\alpha)}\yeqbar^\alpha
			\bigr)
			\leq
			\epsilon_{\rm d}
		\\
			\zeros\leq\ybar^\alpha\leq\alpha\ones
		\\
			-\alpha\ones\leq\yeqbar^\alpha\leq\alpha\ones
		\\
			s(\xbar^\alpha,\ybar^\alpha,\yeqbar^\alpha)\leq\epsilon_{\rm p},
		\end{cases}
	\]
	where
	\begin{equation}\label{eq:sk}
		s(\xbar^\alpha,\ybar^\alpha,\yeqbar^\alpha)
	\coloneqq
		\left\|
			\min\set{
				\begin{pmatrix}
					\ybar^\alpha
				\\
					\alpha\ones-\ybar^\alpha
				\\
					\alpha\ones+\yeqbar^\alpha
				\\
					\alpha\ones-\yeqbar^\alpha
				\end{pmatrix}
			,~
				\begin{pmatrix}
					[\c(\xbar^\alpha)]_-
				\\
					[\c(\xbar^\alpha)]_+
				\\
					[\ceq(\xbar^\alpha)]_-
				\\
					[\ceq(\xbar^\alpha)]_+
				\end{pmatrix}
			}
		\right\|_\infty,
	\end{equation}
	and we say that \((\xbar^\alpha,\ybar^\alpha,\yeqbar^\alpha)\in\R^n\times\R^m\times\R^\meq\) is an \DEF{\eKKTa*-optimal triplet} for \eqref{eq:P}.
\end{definition}

As a next step, we clarify how \eKKT- and \eKKTa-optimality for problem \eqref{eq:P} are interrelated.

\begin{lemma}\label{thm:KKTaKKT}%
	For any \(\vec{\epsilon}=(\epsilon_{\rm p},\epsilon_{\rm d})\geq(0,0)\) the following hold:
	\begin{enumerate}
	\item %
		An \eKKTa-optimal triplet \((\xbar^\alpha,\ybar^\alpha,\yeqbar^\alpha)\) with \(\xbar^\alpha\) \(\epsilon_{\rm p}\)-feasible is also \eKKT-optimal.
	\item \label{thm:KKT=>KKTa}%
		An \eKKT-optimal triplet \((\xbar,\ybar,\yeqbar)\) is also \eKKTa-optimal for any \(\alpha\geq\max\set{\|\ybar\|_\infty,\|\yeqbar\|_\infty}\).
	\end{enumerate}
\end{lemma}

Once again the result is standard, and the proof is obvious by comparing the \eKKT{} and \eKKTa{} optimality conditions, as schematically summarized below:
\[
	\renewcommand{\arraystretch}{1.2}
	\hspace*{-0.5em}
	\rotatebox[origin=c]{90}{\eKKT}
	\begin{cases}
		\dist_{\partial q(\xbar)}\bigl(
			-\trans{\jac \c(\xbar)}\ybar
			-\trans{\jac \ceq(\xbar)}\yeqbar
		\bigr)
		\leq
		\epsilon_{\rm d}
	\\
		\ybar\geq\zeros
	\\
		\|\ceq(\xbar)\|_\infty\leq\epsilon_{\rm p}
	\\
		\|[\c(\xbar)]_+\|_\infty\leq\epsilon_{\rm p}
	\\
		\min\set{\ybar_i,[\c_i(\xbar)]_-}\leq\epsilon_{\rm p}
	\end{cases}
\quad
	\rotatebox[origin=c]{90}{\eKKTa}
	\begin{cases}
		\dist_{\partial\cost(\xbar)}\bigl(
			-\trans{\jac \c(\xbar)}\ybar
			-\trans{\jac \ceq(\xbar)}\yeqbar
		\bigr)
		\leq
		\epsilon_{\rm d}
	\\
		\zeros\leq\ybar\leq\alpha\ones,~
		-\alpha\ones\leq\yeqbar\leq\alpha\ones
	\\
		\min\set{\alpha-\yeqbar_j\sign(\ceq_j(\xbar)),|\ceq_j(\xbar)|}\leq\epsilon_{\rm p}
	\\
		\min\set{\alpha-\ybar_i,[\c_i(\xbar)]_+}\leq\epsilon_{\rm p}
	\\
		\min\set{\ybar_i,[\c_i(\xbar)]_-}\leq\epsilon_{\rm p}
	\end{cases}
\]
with \(i=1,\dots,m\) and \(j=1,\dots,\meq\).

		\subsection{IP-type barrier reformulation}\label{sec:IP}%
			To carry on with the second modification of the problem, in what follows we fix a barrier \(\b\) satisfying the following requirements.

\begin{mybox}
	\begin{assumption}\label{ass:b}%
		The barrier function \(\func{\b}{\R}{\Rinf }\) is proper, lsc, and twice continuously differentiable on its domain \(\dom\b=(-\infty,0)\) with \(\b'>0\) and \(\b''>0\).
	\end{assumption}
\end{mybox}

For reasons that will be elaborated on later, convenient choices of barriers are \(\b(t)=-\frac{1}{t}\), \(\b(t)=\ln(1-\frac1t)\), and the classical logarithmic barrier \(\b(t)=-\ln(-t)\) (all extended as \(\infty\) on \(\R_+\)), see \cref{table:kappa} in \cref{sec:barrierProp}.
Once such \(\b\) is fixed, in the spirit of interior point methods we enforce strict satisfaction of the constraint in \eqref{eq:Qa} by considering the following barrier version
\begin{align*}
	\minimize_{\substack{
			\x\in\R^n,\,\s\in\R^m\\
			\sseq\in\R^\meq
	}}~
	\cost(\x)
&
	+
	\alpha\innprod*{\ones}{\s}
	+\indicator_{\R_+^m}(\s)
	+
	\mu\sum_{i=1}^m\b\bigl(\c_i(\x)-\s_i\bigr)
\\
&
	+
	\alpha\innprod*{\ones}{\sseq}
	+
	\mu\sum_{j=1}^\meq\left[
		\b\bigl(\ceq_j(\x)-\sseq_j\bigr)
		+
		\b\bigl(-\ceq_j(\x)-\sseq_j\bigr)
	\right]
\tag{\({\rm Q}_{\alpha,\mu}\)}\label{eq:Qam}
\end{align*}
for some given parameter \(\mu>0\).
Differently from the IP frameworks of \cite{chouzenoux2020proximal,demarchi2024interior}, we here enforce a barrier in the relaxed version \eqref{eq:Qa}, and \emph{not} on the original problem \eqref{eq:P}.
As such, it is only triplets \((\x,\s,\sseq)\) that need to lie in the interior of the constraints, but \(\x\) is otherwise `unconstrained': for any \(\x\in\R^n\), any \(\s>\c(\x)\) and \(\sseq >|\ceq(\x)|\) (elementwise) yield a triplet \((\x,\s,\sseq)\) that satisfies the strict constraints \(\c(\x)-\s<\zeros\) and \(-\sseq<\ceq(\x)<\sseq\).
Furthermore, notice that the positivity constraint $\s\geq \zeros$ remains untouched, formally imposed by an indicator and not by the barrier. %
In fact, observing that the cost in \eqref{eq:Qam} is separable, we may explicitly minimize with respect to the slack variables \(\s\) and \(\sseq\).
Plugging their optimal values into \eqref{eq:Qam} results in an unconstrained reformulation of the form
\[\tag{\({\rm P}_{\!\alpha,\mu}\)}\label{eq:Pam}
	\minimize_{\x\in\R^n}~
	\cost(\x)
	+
	\mu\Psi\bigl(\c(\x)\bigr)+\mu\Psi^{\rm eq}\bigl(\ceq(\x)\bigr),
\]
where, for any \(\rho*>0\),\footnote{%
	The choice of the starred symbol \(\rho*\) stems from the fact that, as shown in \cref{thm:psi} (see also \cref{fig:b_rho}), this quantity represents a `slope' of \(\b\), that is, a value of its derivative, and we thus treat it as a `dual' object.
}
\begin{equation}\label{eq:Psis_separable}
	\Psi_{\rho*}(\y)
\coloneqq
	\sum_{i=1}^m\psi_{\rho*}(\y_i)
\qquad\text{and}\qquad
	\Psi_{\rho*}^{\rm eq}(\yeq)
\coloneqq
	\sum_{j=1}^\meq\psi_{\rho*}^{\rm eq}(\yeq_j)
\end{equation}
are separable functions with
\begin{align}\label{eq:psidef}%
	\psi_{\rho*}(t)
\coloneqq{} &
	\min_{z\in\R_+}\set{\rho*z+\b(t-z)}
\shortintertext{and}
\label{eq:psieqdef}
	\psi_{\rho*}^{\rm eq}(t)
\coloneqq{} &
	\min_{z\in\R}\set{\rho*z+\b(t-z)+\b(-t-z)}.
\end{align}
These functions satisfy appealing properties summarized in the following theorems.

\begin{theorem}\label{thm:psi}%
	Suppose that \cref{ass:b} holds.
	Then, for any \(\rho*>0\) one has that
	\begin{gather}
	\label{eq:psi}
		\psi_{\rho*}(t)
	=
		\begin{ifcases}
			\b(t) & \b'(t)\leq \rho* \\
			\rho*t-\b*(\rho*) \otherwise
		\end{ifcases}
	\intertext{%
		is convex, Lipschitz differentiable, and \(\rho*\)-Lipschitz continuous with derivative
	}
	\label{eq:psi'}
		\psi_{\rho*}'(t)
	=
		\min\set{\b'(t),\rho*}.
	\end{gather}
	Moreover, for any \(\func{c}{\R^n}{\R}\) convex, the composition \(\psi_{\rho*}\circ c\) is also convex.
\end{theorem}
\begin{proof}
	See \cref{proof:thm:psi}.
\end{proof}

As is apparent from \eqref{eq:psi}, \(\psi_{\rho*}\) coincides with the barrier \(\b\) up to when its slope is \(\rho*\), and after that point it reduces to its tangent line.
As such, \(\psi_{\rho*}\) coincides with a McShane Lipschitz (and globally Lipschitz differentiable) extension \cite{mcshane1934extension} of a portion of the barrier \(\b\), as depicted in \cref{fig:barriers_inequality} and \ref{fig:b_rho}.
	This feature is also evident by viewing \(\psi_{\rho*}\) as the \emph{\(\rho*\)-Pasch-Hausdorff envelope} of \(\b\), as detailed in the proof.

Similar properties are true for \(\psi_{\rho*}^{\rm eq}\), though a corresponding closed-form expression for generic barriers \(\b\) is more cumbersome and not particularly helpful.
An analytic expression is nevertheless available for specific choices of barriers \(\b\), see \cref{table:psi}, or their value at any point can more generally be retrieved at negligible cost by solving a one-dimensional smooth monotone equation.

\begin{theorem}\label{thm:psieq}%
	Suppose that \cref{ass:b} holds.
	Then, for any \(\rho*>0\) one has that
	\begin{gather}
	\label{eq:psieq}
		\psi_{\rho*}^{\rm eq}(t)
	=
		\rho*z_{\rho*}(t)+\b\bigl(t-z_{\rho*}(t)\bigr)+\b\bigl(-t-z_{\rho*}(t)\bigr)
	\intertext{%
		is convex, Lipschitz differentiable, and \(\rho*\)-Lipschitz continuous with derivative
	}
	\label{eq:psieq'}
		(\psi_{\rho*}^{\rm eq})'(t)
	=
		\rho*-2\b'\bigl(-t-z_{\rho*}(t)\bigr)
	\in
		(-\rho*,\rho*),
	\intertext{%
		where, denoting \(\rho\coloneqq\b*'(\rho*)<0\), \(z_{\rho*}(t)>|t|-\rho\) is the unique solution \(z\in\R\) to the smooth monotone equation
	}
	\label{eq:psieq:s*}
		\b'(t-z)
		+
		\b'(-t-z)
	=
		\rho*.
	\end{gather}
	Moreover, for any \(t\neq0\) one has that \(\bigl|(\psi_{\rho*}^{\rm eq})'(t)\bigr|\geq\rho*-2\b'(-|t|)\).
\end{theorem}
\begin{proof}
	See \cref{proof:thm:psieq}.
\end{proof}

	The specific barriers included in \cref{table:psi} are visualized for comparison in \cref{fig:barriers}, along with their corresponding envelopes \(\psi_{\rho^\ast}\) and \(\psi_{\rho^\ast}^{\rm eq}\).
	Notably, the log-like barrier offers an intermediate between the inverse and the logarithmic barriers, bringing together the positive valuedness of the former with the behavior of the latter near $t=0$.
	On the one hand, positive valuedness guarantees via \cref{thm:psi,thm:psieq} that this property is inherited by $\psi_{\rho*}$ and $\psi_{\rho*}^{\rm eq}$, resulting in \eqref{eq:Pam} being more likely well posed.
	On the other hand, we will see below in \cref{sec:barrierProp} that the logarithmic barrier behavior is optimal near \(t=0\), in a certain sense.
	For these reasons, the log-like barrier function is a practical substitute for the classical logarithmic barrier and will be our default choice in the numerical validations,
	which support these claims.

\begin{table}
	\[
	\setlength{\arraycolsep}{10pt}
		\begin{array}{|ccc|}
			\multicolumn{1}{c}{\b(t) \text{ (for \(t<0\))}} & \b*(\tau) \text{ (for \(\tau\geq0\))}& \multicolumn{1}{c}{\s_{\rho*}(t) \text{ (for \(t\in\R\))}}
		\\\hline
			\vphantom{\Bigg|}
			-\frac{1}{t}
			& -2\sqrt{\tau}
			& \sqrt{
				t^2
				+ \frac{1}{\rho*}
				+ \sqrt{\frac{4}{\rho*} t^2 + \frac{1}{(\rho*)^2}}
			}
		\\[2ex]
			\ln\left(1-\frac{1}{t}\right)
			& -2\left(\frac{\sqrt{\tau}}{\sqrt{\tau}+\sqrt{\tau+4}}+\ln\bigl(\frac{\sqrt{\tau}+\sqrt{\tau+4}}{2}\bigr)\right)
			& \sqrt{
				t^2
				+ \frac{1}{4}
				+ \frac{1}{\rho*}
				+ \sqrt{t^2 + \frac{1}{(\rho*)^2} + \frac{4t^2}{\rho*} }
			} - \frac{1}{2}
		\\[2ex]
			-\ln\left(-t\right)
			& -1-\ln(\tau)
			& \frac{1}{\rho*} + \sqrt{ t^2 + \frac{1}{(\rho*)^2} }
		\\[1ex]\hline
		\end{array}
	\]
	\caption[Examples of barriers with their conjugates]{%
		Examples of barriers with their conjugates and analytic expressions for \(\s_{\rho*}(t)\), needed to compute the equality penalty
		\(
			\psi_{\rho*}^{\rm eq}(t)
		\)
		and its derivative
		\(
			(\psi_{\rho*}^{\rm eq})'(t)
		\)
		as in \eqref{eq:psieq} and \eqref{eq:psieq'}.
	}%
	\label{table:psi}%
\end{table}

\begin{figure}[tbh]
	\begin{subfigure}[t]{0.49\linewidth}
		\centering
		\includetikz[width=0.8\linewidth]{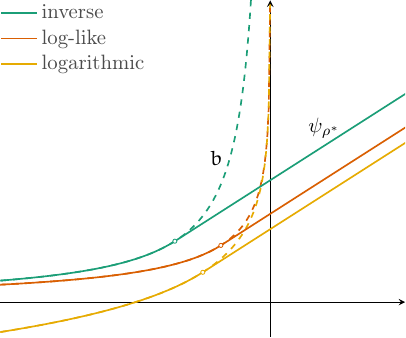}
		\caption[]{%
			Function \(\psi_{\rho*}\) (solid lines) penalizes the violation of the inequality constraint $t\leq 0$.
		}%
		\label{fig:barriers_inequality}%
	\end{subfigure}
	\hfill
	\begin{subfigure}[t]{0.49\linewidth}
		\centering
		\includetikz[width=0.8\linewidth]{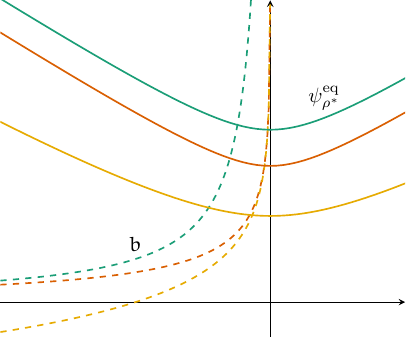}
		\caption[]{%
			Function \(\psi_{\rho*}^{\rm eq}\) (solid lines) penalizes the violation of the equality constraint \(t=0\).%
		}%
		\label{fig:barriers_equality}%
	\end{subfigure}
	\caption[]{%
		Graph of \(\psi_{\rho*}\) (left) and \(\psi_{\rho*}^{\rm eq}\) (right) for different barriers \(\b\) (dashed lines).
		The log-like barrier behaves as an intermediate between the classical inverse and logarithmic barriers.%
	}%
	\label{fig:barriers}%
\end{figure}

The appeal of \(\psi\) and \(\psi^{\rm eq}\) for the sake of addressing problem \eqref{eq:P} lies in their behavior when \(\mu\) is driven to 0, as they respectively approximate the inequality and equality sharp \(L^1\) penalization \(\alpha[{}\cdot{}]_+\) and \(\alpha|{}\cdot{}|\).
In this respect, \cref{fig:b_rho_equality} demonstrates the advantage of our tailored treatment of equality constraints \(\ceq(\x)=\zeros\) (solid lines) as opposed to a naive use of double inequalities \(\pm\ceq(\x)\leq\zeros\) (dotted lines).
Indeed, the latter approach results in the sum
	\begin{equation}\label{eq:psipm}
		t\mapsto\psi^{\pm}(t)\coloneqq \psi(t)+\psi(-t)
	\end{equation}
appearing in the formulation \eqref{eq:Qam}, and because of their opposite slopes around the origin a flat region appears that hinders algorithmic efficiency.
In contrast, the combined marginalization in the definition \eqref{eq:psieq} results in an envelope function \(\psi^{\rm eq}\) that better approximates the sharp \(L^1\) penalty, see also \cref{fig:b_mu_equality}.

\begin{theorem}\label{thm:rho-infty}%
	Suppose that \cref{ass:b} holds.
	Then, \(\psi_{\rho*}/\rho*\to[{}\cdot{}]_+\) and \(\psi_{\rho*}^{\rm eq}/\rho*\to{|{}\cdot{}|}\)  pointwise as \(\rho*\nearrow\infty\).
		Under the assumption that \(\b>0\), both sequences are pointwise decreasing.
	\end{theorem}
\begin{proof}
	See \cref{proof:thm:rho-infty}.
\end{proof}

\begin{figure}[tbh]
	\newcommand{\coeff}{1}%
	\def\xmin{-2}\def\xmax{3}%
	\def\ymin{-0.566}\def\ymax{8.5}%
	\begin{subfigure}[t]{0.49\linewidth}
		\centering
		\includetikz[width=0.8\linewidth]{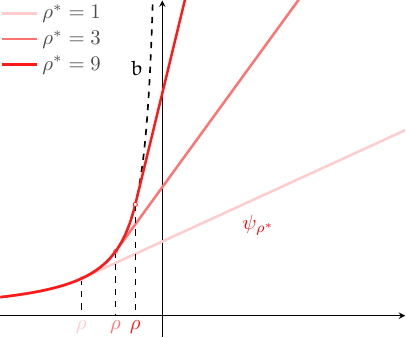}
		\caption[]{%
			Function \(\psi_{\rho*}\) agrees with the barrier \(\b\) until its slope equals \(\rho*\) (at \(\rho\coloneqq\b*'(\rho*)\)), and then continues linearly with slope \(\rho*\).
			Apparently, \(\psi_{\rho*}\nearrow \b\) as \(\rho*\nearrow\infty\).
		}%
		\label{fig:b_rho}%
	\end{subfigure}
	\hfill
	\def\xmin{-2}\def\xmax{3}%
	\def\ymin{-0.966}\def\ymax{14.5}%
	\begin{subfigure}[t]{0.49\linewidth}
		\centering
		\includetikz[width=0.8\linewidth]{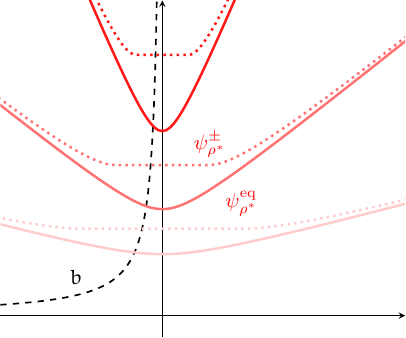}
		\caption[]{%
			Function \(\psi_{\rho*}^{\pm}\) (dotted lines) as in \eqref{eq:psipm} generates a flat region around the origin, due to the opposite slopes \(\rho*\) and \(-\rho*\), which does not appear with \(\psi_{\rho*}^{\rm eq}\) (solid lines).%
		}%
		\label{fig:b_rho_equality}%
	\end{subfigure}
	\caption[]{%
		Graph of \(\psi_{\rho*}\) (left) and \(\psi_{\rho*}^{\rm eq}\) (right) for different values of \(\rho*\).
		These examples employ the inverse barrier \(\b(t)=-\frac{1}{t}+\indicator_{(-\infty,0)}\).
	}%
	\label{fig:b}%
\end{figure}

\begin{figure}[tbh]
	\newcommand{\coeff}{1}%
	\def\xmin{-2}\def\xmax{3}%
	\def\ymin{-0.3}\def\ymax{4.5}%
	\begin{subfigure}[t]{0.48\linewidth}
		\centering
		\includetikz[width=0.8\linewidth]{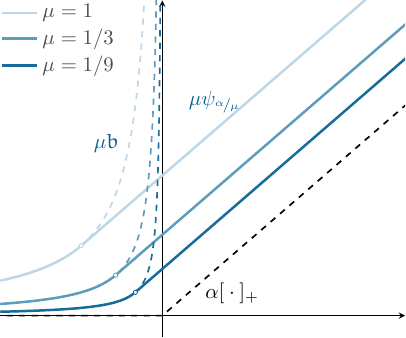}
		\caption[]{%
			As \(\mu\searrow0\), \(\mu\psi\) converges to the sharp \(L^1\) penalty \(\alpha[{}\cdot{}]_+\) while maintaining the same slope \(\alpha\) after the breakpoints.
		}%
		\label{fig:b_mu}%
	\end{subfigure}
	\def\ymin{-0.966}\def\ymax{14.5}%
	\hfill
	\def\ymin{-0.366}\def\ymax{5.5}%
	\begin{subfigure}[t]{0.48\linewidth}
		\centering
		\includetikz[width=0.8\linewidth]{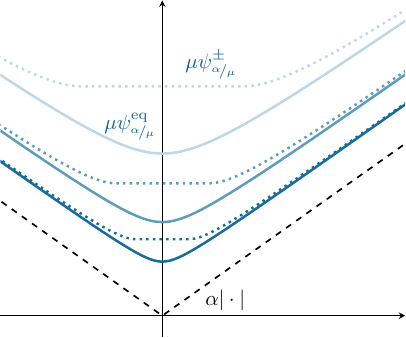}
		\caption[]{%
			As \(\mu\searrow0\), \(\mu\psi^{\rm eq}\) (solid lines) converges to the sharp \(L^1\) penalty \(\alpha|{}\cdot{}|\).
			So does the function \(\mu\psi^{\pm}\) (dotted lines), but with a flat region around zero.%
		}%
		\label{fig:b_mu_equality}%
	\end{subfigure}
	\caption[]{%
		Limiting behavior of \(\mu\psi\) (left) and \(\mu\psi^{\rm eq}\) (right) with constant \(\alpha\) as \(\mu\searrow0\).
		These examples employ the inverse barrier \(\b(t)=-\frac{1}{t}+\indicator_{(-\infty,0)}\).%
	}%
	\label{fig:b_equality}%
\end{figure}

Problem \eqref{eq:Pam} is `unconstrained', in the sense that no explicit ambient constraints are provided, yet stationarity notions relative to it bear a close resemblance with \KKTa-optimality.

\begin{lemma}\label{thm:gradPsi}%
	Suppose that \cref{ass:basic,ass:b} hold.
	Then, for any \(\alpha,\mu>0\) and \(\x\in\R^n\) one has
	\begin{align*}
		\partial\left[\cost+\mu\Psi\circ\c+\mu\Psi^{\rm eq}\circ\ceq\right](\x)
	=
		\partial\cost(\x)
	&{}
		+
		\mu\sum_{i=1}^m \psi'(\c_i(\x))\nabla\c_i(\x)
	\\
	&{}
		+
		\mu\sum_{j=1}^\meq (\psi^{\rm eq})'(\ceq_j(\x))\nabla\ceq_j(\x).
	\end{align*}
	In particular, for any \(\varepsilon\geq0\) a point \(\xbar^{\alpha,\mu}\in\R^n\) is \(\varepsilon\)-stationary for \eqref{eq:Pam} if the pair \((\ybar^{\alpha,\mu},\yeqbar^{\alpha,\mu})\in\R^m\times\R^{\meq}\) given by
	\[
		\ybar^{\alpha,\mu}_i
	\coloneqq
		\mu\psi'(\c_i(\xbar^{\alpha,\mu}))
	\in
		(0,\alpha]
	\quad\text{and}\quad
		\yeqbar^{\alpha,\mu}_j
	\coloneqq
		\mu(\psi^{\rm eq})'(\ceq_j(\xbar^{\alpha,\mu}))
		\in
		(-\alpha,\alpha),
	\]
	\(i=1,\dots,m\),
	\(j=1,\dots,\meq\),
	satisfies
	\[
		\dist_{\partial\cost(\xbar^{\alpha,\mu})}\bigl(
			-\trans{\jac\c(\xbar^{\alpha,\mu})}\ybar^{\alpha,\mu}
			-\trans{\jac\ceq(\xbar^{\alpha,\mu})}\yeqbar^{\alpha,\mu}
		\bigr)
	\leq
		\varepsilon.
	\]
\end{lemma}
\begin{proof}
	The expression of the subdifferential follows from the continuous differentiability of \(\c\) and \(\ceq\) (\cref{ass:basic}), that of \(\psi\) and \(\psi^{\rm eq}\) (\cref{thm:psi,thm:psieq}), and their separable structure as in \eqref{eq:Psis_separable}.
	The inclusion of each element of \(\ybar^{\alpha,\mu}\) in the appropriate interval \((0,\alpha]\) follows from \eqref{eq:psi'}, and similarly the claimed bounds on the components of \(\yeqbar^{\alpha,\mu}\) follow from \eqref{eq:psieq'}.
\end{proof}

The information we gain from \cref{thm:gradPsi} is that whenever \(\xbar\) is \(\varepsilon\)-stationary for \eqref{eq:Pam}, then the triplet \((\xbar,\ybar,\yeqbar)\) with \(\ybar=\mu\Psi'(\c(\xbar))\) and  \(\yeqbar=\mu(\Psi^{\rm eq})'(\ceq(\xbar))\) satisfies all the conditions of \eKKTa[(\epsilon_{\rm p},\varepsilon)] optimality, possibly with the exception of the complementarity slackness involving \(s(\xbar,\ybar,\yeqbar)\).

In conclusion of this section we emphasize that favorable features of the original problem \eqref{eq:P} are likely to be preserved in the formulation \eqref{eq:Qa}.
As expectable, and as explicitly mentioned in \cref{thm:psi,thm:psieq}, convexity is one such property.
More notably, under minimal additional assumptions, whenever \(\c\) and \(\ceq\) are Lipschitz differentiable they remain so after the composition with \(\psi\) and \(\psi^{\rm eq}\).
This consititutes a significant departure from, say, augmented Lagrangian or penalty methods where such property is lost in the composition with quadratic functions.
We exemplify this fact with the following lemma; we note that the statement holds under more general conditions, but a detailed exploration of these broader cases lies beyond the scope of this paper.

\begin{lemma}\label{thm:lipdiff}%
	Suppose that \cref{ass:b} holds, and let \(\func{c}{\R^n}{\R}\) be a Lipschitz-differentiable function.
	Suppose that \(c\) is Lipschitz continuous on the sublevel set \(\set{\x\in\R^n}[c(\x)\leq0]\) (as is the case when it is lower bounded \cite[Lem. 2.3]{hermans2020qpalm}).
	Then, \(\psi_{\rho*}\circ c\) is Lipschitz differentiable for any \(\rho*>0\).
\end{lemma}
\begin{proof}
	Let \(L_c\) and \(\ell\) denote the Lipschitz constant of \(\nabla c\) (on \(\R^n\)) and that of \(c\) on \(\set{\x\in\R^n}[c(\x)\leq0]\), respectively.
	According to \cref{thm:psi}, \(\psi_{\rho*}\) is \(\rho*\)-Lipschitz continuous, coincides with \(\b\) on \((-\infty,\rho]\), and is then linear with slope \(\rho*\) on \((\rho,\infty)\), where \(\rho\coloneqq\b*'(\rho*)<0\).
	Fix \(\x,\y\in\R^n\), and without loss of generality assume that \(c(\x)\leq c(\y)\).
	We have
	\begin{align*}
		\|\nabla(\psi_{\rho*}\circ c)(\x)-\nabla(\psi_{\rho*}\circ c)(\y)\|
	={} &
		\|\psi_{\rho*}'(c(\x))\nabla c(\x)-\psi_{\rho*}'(c(\y))\nabla c(\y)\|
	\\
	\leq{} &
		\psi_{\rho*}'(c(\x))\|\nabla c(\x)-\nabla c(\y)\|
		+
		\|\nabla c(\y)\||\psi_{\rho*}'(c(\x))-\psi_{\rho*}'(c(\y))|
	\\
	\leq{} &
		\rho* L_c\|\x-\y\|
		+
		\|\nabla c(\y)\|\bigl(\psi_{\rho*}'(c(\y))-\psi_{\rho*}'(c(\x))\bigr).
	\end{align*}
	It remains to account for the second term in the last sum.
	If \(c(\x)\leq c(\y)\leq\rho\), then \(\psi_{\rho*}\) coincides with \(\b\) in all occurrences, and the term can be upper bounded as
	\(
		B\ell^2
		\|\x-\y\|
	\),
	where \(B\coloneqq\max_{(-\infty,\rho]}\b''\) is a Lipschitz modulus for \(\b'\) on \((-\infty,\rho]\).
	If \(c(\x)\leq\rho<c(\y)\), then \(\psi_{\rho*}'(c(\y))=\psi_{\rho*}'(\rho)\) and, by continuity, there exists \(t\in[0,1]\) such that \(c(\x+t(\y-\x))=\rho\), so that \(\psi_{\rho*}'(c(\x+t(\y-\x)))=\psi_{\rho*}'(\rho)\), resulting in the same bound
	\(
		Bt\ell^2
		\|\x-\y\|
	\leq
		B\ell^2
		\|\x-\y\|
	\).
	Lastly, if \(\rho\leq c(\x)\leq c(\y)\) then the last term is zero.
	In all cases we conclude that
	\[
		\|\nabla(\psi_{\rho*}\circ c)(\x)-\nabla(\psi_{\rho*}\circ c)(\y)\|
	\leq
		\bigl(\tfrac\alpha\mu L+B\ell^2\bigr)
		\|\x-\y\|
	\quad
		\forall\x,\y\in\R^n,
	\]
	proving the claim.
\end{proof}

	\section{Algorithmic framework}\label{sec:Algorithm}%
		The main ingredient for the proposed numerical scheme is the penalty-barrier problem \eqref{eq:Pam}.
As shown in the previous section, the cost function in \eqref{eq:Pam} converges pointwise to the original hard-constrained cost \(\cost+\indicator_{\R_-^m}\circ\c+\indicator_{\set{\zeros}}\circ\ceq\) of \eqref{eq:P} as \(\mu\searrow0\) and \(\alpha\nearrow\infty\).
Following a homotopic rationale, this motivates solving (up to approximate local optimality) instances of \eqref{eq:Pam} for progressively small values of \(\mu\) and larger values of \(\alpha\).
This is the leading idea of the algorithmic framework of \cref{alg:general} presented in this section, whose name `\pippo' evokes the key underlying feature of \emph{margi}nalization discussed in \cref{sec:IP}.
The update rules for the coefficients are carefully designed so as to ensure that the output satisfies suitable optimality conditions for the original problem \eqref{eq:P}, as well as to prevent the \(L^1\) penalization parameter \(\alpha\) in \eqref{eq:Pam} from divergent behaviors under favorable conditions on the problem.
This is the reason behind the involvement of the conjugate \(\b*\) in the update criterion at \cref{state:general:pk>eps}, as will be revealed in \cref{sec:barrierProp,sec:convex} through a systematic study of the properties of the barrier \(\b\) in the generality of \cref{ass:basic} as well as when specialized to the convex case.

\begin{algorithm}
	\caption{\protect\pippo*: A combined penalty and barrier framework for constrained optimization}
	\label{alg:general}
	\begin{algorithmic}[1]
\item[\algfont{require}]
	tolerances \(\epsilon_{\rm p},\epsilon_{\rm d}\geq0\);~
	parameters \(\alpha_0,\mu_0>0\), \(\varepsilon_0\geq\epsilon_{\rm d}\),
	\(\delta_\alpha>1\) and \(\delta_\varepsilon,\delta_\mu\in(0,1)\)%

\itemsep=0.75ex
\item[\algfont{repeat for} \(k=0,1,\dots\)]

\itemsep=0.9ex
\State \label{state:general:xk}%
	Find an \(\varepsilon_k\)-stationary point \(\x^k\) for \eqref{eq:Pam} with \((\alpha,\mu)=(\alpha_k,\mu_k)\)

\State \label{state:general:yk}%
	set
	\(
		\y_i^k
	=
		\mu_k\psi_{\nicefrac{\alpha_k}{\mu_k}}'(\c_i(\x^k))
	\)
	as in \eqref{eq:psi'},
	\(i=1,\dots,m\)

\State \label{state:general:yeqk}%
	set
	\(
		\yeq_j^k
	=
		\mu_k(\psi_{\nicefrac{\alpha_k}{\mu_k}}^{\rm eq})'(\ceq_j(\x^k))
	\)
	as in \eqref{eq:psieq'},
	\(j=1,\dots,\meq\)

\State \label{state:general:pk}%
	\(
		p_k
	=
		\max\set{\|[\c(\x^k)]_+\|_\infty, \|\ceq(\x^k)\|_\infty}
	\)
\Comment{constraints violation}

\State \label{state:general:sk}%
	\(
		s_k=s(\x^k,\y^k,\yeq^k)
	\)
	as in \eqref{eq:sk}
\Comment{complementarity violation}

\If{
	\((\varepsilon_k,p_k,s_k) \leq (\epsilon_{\rm d},\epsilon_{\rm p},\epsilon_{\rm p})\)
}
	\Statex*
		\algfont{return} \((\x^k,\y^k)\) \eKKT[(\epsilon_{\rm p},\epsilon_{\rm d})] pair for \eqref{eq:P}
\EndIf\vspace{-1.5ex}%

\State
	\(
	\varepsilon_{k+1}
	=
	\max\set{\delta_\varepsilon\varepsilon_k,\,\epsilon_{\rm d}}
	\)\label{state:general:update_eps}

\State \label{state:general:pk>eps}%
	\algfont{if}~
	\(
		p_k
	>
		\max\set{
			\epsilon_{\rm p},
			2(m+\meq)
			\frac{-\b*(\nicefrac{\alpha_k}{\mu_k})}{\nicefrac{\alpha_k}{\mu_k}}
		}
	\)
	~\algfont{then}~
	\(\alpha_{k+1}=\delta_\alpha\alpha_k\),
	~\algfont{else}~
	\(\alpha_{k+1}=\alpha_k\)%

\setcounter{copycounter}{\value{ALG@line}}%
\State \label{state:general:update_mu}%
	\algfont{if}~
	\(s_k>\epsilon_{\rm p}\)
	~\algfont{or}~
	\(\alpha_{k+1}=\alpha_k\)
	~\algfont{then}~
	\(\mu_{k+1}=\delta_\mu\mu_k\),
	~\algfont{else}~
	\(\mu_{k+1}=\mu_k\)%
\end{algorithmic}

\end{algorithm}

\Cref{alg:general} is not tied to any particular solver for addressing each instance of \eqref{eq:Pam} at \cref{state:general:xk}.
Whenever \(q\) amounts to the sum of a differentiable and a prox-friendly function (in the sense that its proximal mapping is easily computable),
such structure is retained by the cost function in \eqref{eq:Pam}, indicating that proximal-gradient based methods are suitable candidates.
This was also the case in the purely interior-point based \ipprox{} of \cite{demarchi2024interior}, which considers a plain proximal gradient with a backtracking routine for selecting the stepsizes.
Differently from the subproblems of \ipprox{} in which the differentiable term is extended-real valued, the differentiable term in \eqref{eq:Pam} is smooth \emph{on the whole \(\R^n\)}.
This enables the employment of more sophisticated proximal-gradient-type algorithms such as \panocp{} \cite{stella2017simple,demarchi2022proximal} that make use of higher-order information to considerably enhance convergence speed.
This claim will be substantiated with numerical evidence in \cref{sec:numerical}; in this section, we instead focus on properties of the \emph{outer} \cref{alg:general} that are independent of the \emph{inner} solver.

\begin{remark}\label{rem:wellposedsubproblems}%
	Throughout our convergence analysis, it is assumed that \cref{alg:general} is well-defined, thus requiring that each subproblem \eqref{eq:Pam} at \cref{state:general:xk} admits an approximate stationary point.
	Moreover, some of the following statements assume the existence of an accumulation point \(\bar{\x}\) for a sequence \(\seq{\x^k}\) generated by \cref{alg:general}.
	In general, these preconditions can be verified with coercivity or (level) boundedness arguments; for instance, the objective function of \eqref{eq:Pam} is bounded from below whenever \(\dom q\) is a compact set.
\end{remark}

\begin{remark}[parameter update variant]\label{rem:variant}%
	According to \cref{state:general:pk>eps,state:general:update_mu}, at each iteration either \(\alpha_k\) or \(\mu_k\) (possibly both) is updated.
	With the aim of slowing down the reduction of \(\mu_k\) to attenuate unnecessary ill-conditioning,
	another viable option is to restrict the update condition as in
	\begin{algorithmic}[1]%
	\makeatletter
		\renewcommand{\alglinenumber}[1]{\footnotesize\textbf{\ref*{alg:general}}.\oldstylenums{\arabic{ALG@line}}':}%
		\setcounter{ALG@line}{\value{copycounter}}%
	\makeatother
	\State \label{state:general:update_mu'}%
		\algfont{if}~
		\(s_k>\epsilon_{\rm p}\)
		~\algfont{or}~
		\((\alpha_{k+1},\varepsilon_{k+1})=(\alpha_k,\varepsilon_k)\)
		~\algfont{then}~
		\(\mu_{k+1}=\delta_\mu\mu_k\),
		~\algfont{else}~
		\(\mu_{k+1}=\mu_k\)%
	\end{algorithmic}
	allowing also the possibility of neither parameter being updated.
	Such circumstance takes place only if \(\varepsilon_k>\epsilon_{\rm d}\), owing to \cref{state:general:update_eps}.
	In this event, then,
	it is only the stationarity tolerance \(\varepsilon_k\) for \cref{state:general:xk} that is decreased, and the next iteration reduces to solving the same subproblem with higher accuracy.
	To avoid unnecessary notational complexity in the proofs, we adhere to the steps outlined in \cref{alg:general}, while noting that the entire theoretical framework remains valid for this variant, with only minor changes required to the iteration indexing.
	For our numerical tests in \cref{sec:numerical}, however, \cref{state:general:update_mu} of \cref{alg:general} is replaced by the variant above.
\end{remark}

		\subsection{Convergence analysis}
			\begin{lemma}[properties of the iterates]%
	Suppose that \cref{ass:basic,ass:b} hold, and consider the iterates generated by \cref{alg:general}.
	At every iteration \(k\) the following hold:
	\begin{enumerate}
	\item \label{thm:DF}%
		\(\zeros<\y^k\leq\alpha_k\ones\) and \(|\yeq^k|<\alpha_k\ones\).

	\item \label{thm:PO}%
		(\(\x^k\in\dom\cost\) and) \(\dist_{\partial\cost(\x^k)}(-\trans{\jac\c(\x^k)}\y^k-\trans{\jac\ceq(\x^k)}\yeq^k)\leq\varepsilon_k\).

	\item \label{thm:CS}%
		If (\(\epsilon_{\rm p}>0\) and) \(\mu_k\leq\frac{\epsilon_{\rm p}}{2\b'(-\epsilon_{\rm p})}\), then \(s_k\leq\epsilon_{\rm p}\).

	\item \label{thm:rho}%
		For \(k\geq1\), either \(\alpha_k=\delta_\alpha\alpha_{k-1}\) or \(\mu_k=\delta_\mu\mu_{k-1}\) (possibly both);
		in particular, letting \(\rho*_k\coloneqq\nicefrac{\alpha_k}{\mu_k}\) and \(\delta_{\rho*}\coloneqq\min\set{\delta_\alpha,\delta_\mu^{-1}}\) it holds that \(\rho*_k\geq\delta_{\rho*}\rho*_{k-1}\).
	\end{enumerate}
\end{lemma}
\begin{proof}
	Assertions \ref{thm:DF} and \ref{thm:PO} follow from \cref{thm:gradPsi}.
	Assertion \ref{thm:rho} is obvious by observing that whenever \(\alpha_{k+1}=\alpha_k\) the update \(\mu_{k+1}=\delta_\mu\mu_k\) is enforced.

	We finally turn to assertion \ref{thm:CS}, and suppose that \(\mu_k\leq\frac{\epsilon_{\rm p}}{2\b'(-\epsilon_{\rm p})}\).
	Then, for all \(i\) such that \([\c_i(\x^k)]_->\epsilon_{\rm p}\) (or, equivalently, \(\c_i(\x^k)<-\epsilon_{\rm p}\)), one has
	\[
		\y^k_i
	=
		\mu_k\psi_{\nicefrac{\alpha_k}{\mu_k}}'(\c_i(\x^k))
	\leq
		\mu_k\b'(\c_i(\x^k))
	\leq
		\mu_k\b'(-\epsilon_{\rm p})
	\leq
		{\tfrac{1}{2}}\epsilon_{\rm p}
	\leq
		\epsilon_{\rm p},
	\]
	where the first inequality follows from the definition of \(\y^k\) together with \eqref{eq:psi'}, and the second one owes to monotonicity of \(\b'\).
	On the other hand, for all \(i\) such that \(\c_i(\x^k)>\epsilon_{\rm p}\) (in fact, more generally when \(\c_i(\x^k)>0\)), one has that
	\(
		\y^k_i
	=
		\mu_k\psi_{\nicefrac{\alpha_k}{\mu_k}}'(\c_i(\x^k))
	=
		\alpha_k
	\),
	cf. \eqref{eq:psi'}.
	Next, if \(j\) is such that \(\ceq_j(\x^k)>\epsilon_{\rm p}\), one has that
	\[
		\yeq^k_j
	=
		\mu_k(\psi_{\nicefrac{\alpha_k}{\mu_k}}^{\rm eq})'(\ceq_j(\x^k))
	\geq
		\mu_k\bigl(\nicefrac{\alpha_k}{\mu_k}-2\b'(-\epsilon_{\rm p})\bigr)
	\geq
		\alpha-\epsilon_{\rm p},
	\]
	where the first inequality follows from \cref{thm:psieq}.
	By the same arguments, we deduce that
	\(
		\yeq^k_j
	\leq
		-\alpha+\epsilon_{\rm p}
	\)
	holds for all \(j\) such that \(\ceq_j(\x^k)<-\epsilon_{\rm p}\).
	In summary, for all \(i=1,\dots,m\) at least one among \(\y_i^k\) and \([\c_i(\x^k)]_-\) is not larger than \(\epsilon_{\rm p}\), and at least one among \(\alpha-\y_i^k\) and \([\c_i(\x^k)]_+\) is zero.
	Similarly, for all \(j=1,\dots,\meq\) at least one among \(\alpha-\yeq_j^k\sign(\ceq_j(\x^k))\) and \(|\ceq_j(\x^k)|\) is not larger than \(\epsilon_{\rm p}\),
	overall proving that \(s_k\leq\epsilon_{\rm p}\).
\end{proof}

\begin{corollary}[stationarity of feasible limit points]\label{thm:feas=>AKKT}%
	Let \cref{ass:basic,ass:b} hold, and consider the iterates generated by \cref{alg:general}.
	If the algorithm runs indefinitely, then \(-\frac{\mu_k}{\alpha_k}\b*(\nicefrac{\alpha_k}{\mu_k})\searrow0\) as \(k\to\infty\).
	Moreover, any feasible accumulation point of \(\seq{\x^k}\) that belongs to \(\dom\cost\) is \AKKT-optimal for \eqref{eq:P}.
\end{corollary}
\begin{proof}
	The monotonic vanishing of \(-\frac{\mu_k}{\alpha_k}\b*(\nicefrac{\alpha_k}{\mu_k})\) follows from \cref{thm:b*/t*,thm:rho}.
	Suppose that \(\seq{\x^k}[k\in K]\to\xbar\in\dom\cost\) with \(\c(\xbar)\leq\zeros\) and \(\ceq(\xbar)=\zeros\), and let \(i\) be such that \(\c_i(\xbar)<0\) (if such an \(i\) does not exist, then there is nothing to show).
	According to \cref{defin:AKKT,thm:ykck}, it suffices to show that \(\seq{\y_i^k}[k\in K]\to0\); in turn, by definition of \(\y^k\) and continuity of \(\c\) it suffices to show that \(\mu_k\searrow0\).
	If \(\epsilon_{\rm p}>0\), then continuity of \(\c\) implies that \(\alpha_{k+1}=\alpha_k\) for all \(k\in K\) large enough, hence, by virtue of \cref{thm:rho}, \(\mu_{k+1}=\delta_\mu\mu_k\) for all such \(k\).
	Since \(\seq{\mu_k}\) is monotone, in this case \(\mu_k\searrow0\) as \(k\to\infty\).
	Suppose instead that \(\epsilon_{\rm p}=0\), and, to arrive to a contradiction, that \(\mu_k\) is asymptotically constant.
	This implies that \(s_k\leq\epsilon_{\rm p}=0\) eventually always holds, which is a contradiction since
	\[
		s_k
	\geq
		\min\set{\y^k_i,[\c_i(\x^k)]_-}
	\geq
		\min\set{\y^k_i,-\tfrac12\c_i(\xbar)}
	>
		0
	\quad
		\forall k\in K\text{ large,}
	\]
	where the first inequality follows by definition of \(s_k\), cf. \cref{state:general:sk}, the second one for \(k\in K\) large since \(\c_i(\x^k)\to\c_i(\xbar)<0\) as \(K\ni k\to\infty\), and the last one because \(\y^k>\zeros\).
\end{proof}

The update rule for the penalty parameter does not demand (approximate) feasibility, but it depends on a relaxed condition at \cref{state:general:pk>eps}.
By \eqref{eq:b*/t*} in the appendix, the second term vanishes as $\alpha/\mu \to\infty$, so the penalty parameter is eventually increased as needed to achieve $\epsilon_{\rm p}$-feasibility.
The relaxation of this condition using a quantity involving the conjugate \(\b*\) mitigates the growth of \(\alpha\).
Simultaneously, under suitable choices of the barrier \(\b\), it ensures that this parameter remains unchanged only if the constraints violation stays within a controlled range, as will be ultimately demonstrated in \cref{thm:linearsk}.

\begin{theorem}\label{thm:outcomes}
	Suppose that \cref{ass:basic,ass:b} hold, and consider the iterates generated by \cref{alg:general} with \(\epsilon_{\rm p},\epsilon_{\rm d}>0\).
	Then, \(\inf_{k\in\N}\mu_k>0\) and exactly one of the following scenarios occurs:%
	\begin{enumerator}
	\item
		either the algorithm terminates returning an \eKKT[(\epsilon_{\rm p},\epsilon_{\rm d})] stationary point for \eqref{eq:P},
	\item
		or it runs indenfinitely with \(s_k\leq\epsilon_{\rm p}<p_k\) for all \(k\) large enough, and \(\seq{\alpha_k}\nearrow\infty\).
	\end{enumerator}
	In the latter case, if \(\dom\cost\) is closed and \(\cost\) is continuous relative to it, then for any accumulation point \(\xbar\) of \(\seq{\x^k}\) one has that \((\xbar,\cost(\xbar))\) is KKT-stationary for the feasiblity problem
	\begin{equation}\label{eq:minInfeas}
		\minimize_{(\x,t)\in\epi\cost}~\|[\c(\x)]_+\|_1+\|\ceq(\x)\|_1,
	\end{equation}
	in the sense that
	\(
		(\zeros,0)
	\in
		\partial\bigl[\|[\c (\xbar)]_+\|_1+\|\ceq(\xbar)\|_1\bigr]
		\times
		\set{0}
		+
		\ncone_{\epi\cost}(\xbar,\cost(\xbar))
	\).
\end{theorem}
\begin{proof}
	Since \(\mu_{k+1}\leq\mu_k\) for all \(k\), and \(\mu_k\) is linearly reduced whenever \(s_k>\epsilon_{\rm p}\), we conclude that (either the algorithm terminates or) \(s_k\leq\epsilon_{\rm p}\) eventually always holds.

	If the algorithm returns \((\x^k,\y^k,\yeq^k)\), then the compliance with the termination criteria combined with the fact that \(\y^k>\zeros\) for all \(k\), see \cref{thm:DF}, ensures that such triplet meets all conditions in \cref{defin:eKKT}, and hence it is \eKKT[(\epsilon_{\rm p},\epsilon_{\rm d})]-stationary for \eqref{eq:P}.

	Suppose instead that the algorithm does not terminate.
	Clearly, \(\varepsilon_k=\epsilon_{\rm d}\) holds for \(k\) large enough, so that the only unmet termination criterion is eventually \(p_k\leq\epsilon_{\rm d}\).
	Therefore, \(p_k>\epsilon_{\rm p}\) holds for every \(k\) large enough.
	It follows from \cref{thm:rho,thm:b*/t*} that \(-2(m+\meq)\b*(\rho*_k)/\rho*_k\) eventually drops below \(\epsilon_{\rm p}\), implying that the condition for increasing \(\alpha_{k+1}\) at \cref{state:general:pk>eps} reduces to \(p_k>\epsilon_{\rm p}\).
	Having shown that this is eventually always the case, \(\alpha_{k+1}=\delta_\alpha\alpha_k\) always holds for \(k\) large, \(\alpha_k\nearrow\infty\), and \(\mu_k\) is eventually never updated, cf. \cref{state:general:pk>eps}.

	To conclude, suppose that \(\dom\cost\) is closed and that \(\cost\) is continuous relative to this set.
	By \cref{thm:gradPsi}, for every \(k\) we have that there exists \(\vec{\eta}^k\in\R^n\) with \(\|\vec{\eta}^k\|\leq\epsilon_{\rm d}\) such that
	\[
		\vec{\eta}^k
		-
		\trans{\jac\c(\x^k)}\y^k
		-
		\trans{\jac\ceq(\x^k)}\yeq^k
	\in
		\partial\cost(\x^k).
	\]
	Let \(\xbar\) be the limit of a subsequence \(\seq{\x^k}[k\in K]\) and, up to extracting, let \(\lambar\) and \(\lameqbar\) be the limits of \(\seq{\frac{1}{\alpha_k}\y^k}[k\in K]\) and \(\seq{\frac{1}{\alpha_k}\yeq^k}[k\in K]\), respectively.
	The definition of \(\y^k\) and \(\yeq^k\) together with the continuity of \(\c\) and \(\ceq\) yields that
	\[
		\lambar_i
		\begin{ifcases}
			=0 & \c_i(\xbar)<0\\
			=1 & \c_i(\xbar)>0\\
			\in[0,1] & \c_i(\xbar)=0
		\end{ifcases}
	\qquad\text{and}\qquad
		\lameqbar_j
		\begin{ifcases}
			=-1 & \ceq_j(\xbar)<0\\
			=1 & \ceq_j(\xbar)>0\\
			\in[-1,1] & \ceq_j(\xbar)=0
		\end{ifcases}
	\]
	or, equivalently,
	\begin{equation}\label{eq:lambdasubdiff}
		\lambar_i
	\in
		\partial[{}\cdot{}]_+(\c_i(\xbar))
	\quad\text{and}\quad
		\lameqbar_j
	\in
		\partial\|{}\cdot{}\|_1(\ceq_j(\xbar))
	\end{equation}
	for every \(i=1,\dots,m\) and \(j=1,\dots,\meq\).
	Since \(\dom\cost\) is closed and \(\x^k\in\dom\cost\) for all \(k\), one has that \(\cost(\xbar)<\infty\).
	Moreover, it follows from the continuity assumption that \(\cost(\x^k)\to\cost(\xbar)\) as \(K\ni k\to\infty\), hence that
	\(
		-
		\trans{\jac\c(\xbar)}\lambar
		-
		\trans{\jac\ceq(\xbar)}\lameqbar
	\in
		\partial^\infty\cost(\xbar)
	,
	\)
	where \(\partial^\infty\cost(\xbar)\) denotes the horizon subdifferential of \(\cost\) at \(\xbar\).
	Appealing to \cite[Thm. 8.9 and Ex. 8.14]{rockafellar1998variational}, this means that
	\begin{equation}\label{eq:episubdiff}
		\bigl(
			-
			\trans{\jac\c(\xbar)}\lambar
			-
			\trans{\jac\ceq(\xbar)}\lameqbar
			,
			0
		\bigr)
	\in
		\ncone_{\epi\cost}(\xbar,\cost(\xbar))
	=
		\partial\indicator_{\epi\cost}(\xbar,\cost(\xbar)).
	\end{equation}
	Since
	\[
		\partial\|[\c(\x)]_+\|_1
	=
		\sum_{i=1}^m\partial[\c_i(\x)]_+
	=
		\sum_{i=1}^m\partial[{}\cdot{}]_+(\c_i(\x))\nabla\c_i(\x)
	\]
	and similarly
	\[
		\partial\|\ceq(\x)\|_1
	=
		\sum_{j=1}^{\meq}\partial|\ceq_j(\x)|
	=
		\sum_{j=1}^{\meq}\partial|{}\cdot{}|(\ceq_j(\x))\nabla\ceq_j(\x),
	\]
	see \cite[Ex. 10.26]{rockafellar1998variational}, it follows from \eqref{eq:lambdasubdiff} and continuity of \(\|[\c(\x)]_+\|_1\) and \(\|\ceq(\x)\|_1\) that \(\lambar\in\partial\|[\c(\xbar)]_+\|_1\) and \(\lameqbar\in\partial\|\ceq(\xbar)\|_1\).
	Combining with \eqref{eq:episubdiff} concludes the proof.
\end{proof}

\begin{remark}[relaxing continuity of {$\protect\cost$}]\label{thm:noC0}%
	Similarly to the commentary after \cref{thm:ykbounded} pertaining the \emph{limiting} subdifferential \(\partial\cost\), note that continuity of \(\cost\) on its domain is used to guarantee that
	equality \(\partial^\infty\cost(\xbar)=\tilde\partial^\infty\cost(\xbar)\) (as opposed to inclusion \(\partial^\infty\cost(\xbar)\subseteq\tilde\partial^\infty\cost(\xbar)\)) holds for the \emph{horizon} subdifferential \(\partial^\infty\cost(\xbar)\) of \(\cost\) at any point \(\xbar\in\dom\cost\), where for reference we denote
	\[
		\tilde\partial^\infty\cost(\xbar)
	\coloneqq
		\set{\bar{\vec{v}}\in\R^n}[
			\exists\lambda_k\to0,~ (\x^k,\vec{v}^k)\in\graph\hat\partial\cost
			\text{ such that }
			(\x^k,\lambda_k\vec{v}^k)\to(\xbar,\bar{\vec{v}})
		].
	\]
	Here, the definition of \(\tilde\partial^\infty\cost(\xbar)\) matches that of \(\partial^\infty\cost(\xbar)\) except that the \emph{``\(\cost\)-attentive''} constraint \(\cost(\x^k)\to\cost(\xbar)\) is not imposed for the sequence \(\seq{\x^k}\).
	The validity of \cref{thm:outcomes} is thus unaffected if continuity of \(\cost\) on its domain is replaced by \(\partial^\infty\cost\) being as above at any point \(\xbar\in\dom\cost\), and once again this generalization covers important examples such as functions involving \(L^0\)-norm terms.

	On the other hand, note that mere outer semicontinuity of \(\partial\cost\) as in \cref{thm:ykbounded} is not enough to ensure this identity.
	To see this, consider the function \(\func{h}{\R}{\R}\) defined as \(h(x)=0\) for \(x\leq0\) \(h(x)=1-\sqrt{x}\) for \(x>0\).
	One can easily see that \(\hat\partial h=\partial h\) with \(\partial h(x)=\set{\nabla h(x)}\) for \(x\neq 0\) and \(\partial h(0)=[0,\infty)\) is everywhere osc, yet \(\partial^\infty\cost(0)=[0,\infty)\neq\R=\tilde\partial^\infty\cost(0)\).
\end{remark}

The abuse of terminology to express KKT-stationarity in terms of subdifferentials passes through the same construct relating \eqref{eq:Pa} and \eqref{eq:Qa}, in which a slack variable is tacitly introduced to reformulate the \(L^1\) norm; see the discussion in \cref{sec:L1}.
More importantly, the involvement in \eqref{eq:minInfeas} of the epigraph of \(\cost\), as opposed to its domain, is a necessary technicality that cannot be avoided in the generality of \cref{ass:basic}, as we illustrate next.

\begin{remark}[{$\protect\epi\protect\cost$ vs $\protect\dom\protect\cost$}]%
	Stationarity for \eqref{eq:minInfeas} is, in general, weaker than that for the more natural minimal infeasibility violation problem
	\begin{equation}\label{eq:minInfeasDom}
		\minimize_{\x\in\dom\cost}~\|[\c(\x)]_+\|_1+\|\ceq(\x)\|_1.
	\end{equation}
	To see how this notion may be violated, consider \(\cost(x)=\sqrt{|x|}\) and \(c(x)=x+1\), so that \eqref{eq:P} reads
	\[
		\minimize_{x\in\R}~\sqrt{|x|}
	\quad
		\stt{}~x\leq-1.
	\]
	The point \(x^k=0\) is stationary for any subproblem \eqref{eq:Pam} with arbitrary \(\alpha,\mu>0\), and therefore constitutes a feasible choice in \cref{alg:general}.
	However, the limit \(\bar x=0\) of the corresponding constant sequence is not stationary for the minimization of \([x+1]_+\) over \(\dom\cost=\R\).
	Nevertheless,
	\[
		\partial[x+1]_+(0)\times\set{0}+\ncone_{\epi\cost}(0,0)
	=
		\set{(1,0)}+(\R\times\set{0})
	\ni
		(0,0),
	\]
	confirming that \((0,0)\) is stationary for the epigraphical problem \eqref{eq:minInfeas}.
\end{remark}

We next formally illustrate why stationarity for \eqref{eq:minInfeasDom} always implies that for \eqref{eq:minInfeas}, and identify the culprit of a possible discrepancy in uncontrolled growths around \(\xbar\) from within \(\dom\cost\).
To this end, we remind that a proper function \(\func{h}{\R^n}{\Rinf}\) is said to be \DEF{calm} at a point \(\xbar\in\dom h\) relative to a set \(X\ni\xbar\) if
\[
	\liminf_{\substack{X\ni\x\to\xbar\\ \x\neq\xbar}}
	\frac{|h(\x) - h(\xbar)|}{\|\x-\xbar\|}
<
	\infty,
\]
and that this condition is weaker than strict continuity.

\begin{lemma}\label{thm:epiLip}%
	Let \(\func{h}{\R^n}{\Rinf}\) be proper and lsc.
	Then, for any \(\xbar\in\dom h\) one has
	\begin{align}
		\widehat\ncone_{\dom h}(\xbar)
	\subseteq{} &
		\set{\vec{\bar v}\in\R^n}[(\vec{\bar v},0)\in\widehat\ncone_{\epi h}(\xbar,h(\xbar))]
	\nonumber
	\shortintertext{and}
		\ncone_{\dom h}(\xbar)
	\subseteq{} &
		\set{\vec{\bar v}\in\R^n}[(\vec{\bar v},0)\in\ncone_{\epi h}(\xbar,h(\xbar))]
		=
		\partial^\infty h(\xbar).
		\label{eq:NdomSubset}
	\end{align}
	When \(h\) is convex, both inclusions hold as equality.
	More generally, when \(h\) is calm (in particular, if it is strictly continuous) at \(\xbar\) relative to \(\dom h\), then the first inclusion holds as equality, and so does the second one when such property holds not only at \(\xbar\), but also at all points in \(\dom h\) close to it.
\end{lemma}
\begin{proof}
	The relations in the convex case are shown in \cite[Thm. 8.9 and Prop. 8.12]{rockafellar1998variational}; in what follows, we consider an arbitrary proper and lsc function \(h\).
	Let \(\bar v\in\widehat\ncone_{\dom h}(\xbar)\) and let \(\epi h\ni(\x^k,t^k)\to(\xbar,h(\xbar))\).
	Then, there exists \(\varepsilon_k\to0\) such that
	\(
		\innprod{\vec{\bar v}}{\x^k-\xbar}
	\leq
		\varepsilon_k\|\x^k-\xbar\|
	\)
	holds for every \(k\), hence
	\[
		\varepsilon_k
		\left\|\binom{\x^k-\xbar}{t^k-h(\xbar)}\right\|
	\geq
		\varepsilon_k
		\|\x^k-\xbar\|
	\geq
		\innprod{\vec{\bar v}}{\x^k-\xbar}
	=
		\innprod{\binom{\vec{\bar v}}{0}}{\binom{\x^k-\xbar}{t^k-h(\xbar)}}.
	\]
	By the arbitrariness of the sequence, we conclude that \((\vec{\bar v},0)\in\widehat\ncone_{\epi h}(\xbar,h(\xbar))\).
	The same inclusion must then hold for the limiting normal cones, leading to
	\eqref{eq:NdomSubset},
	where the identity follows from \cite[Thm. 8.9]{rockafellar1998variational}.

	Suppose now that there exists \(\kappa>0\) such that \(|h(\x)-h(\xbar)|\leq \kappa\|\x-\xbar\|\) for \(\x\in\dom h\) close to \(\xbar\), and suppose that \((\vec{\bar v},0)\in\widehat\ncone_{\epi h}(\xbar,h(\xbar))\).
	Let \(\dom h\ni\x^k\to\xbar\), and note that \(\epi h\ni(\x^k,h(\x^k))\to(\xbar,h(\xbar))\).
	Then, there exists \(\varepsilon_k\to0\) such that
	\begin{align*}
		\innprod{\vec{\bar v}}{\x^k-\xbar}
	=
		\innprod{\binom{\vec{\bar v}}{0}}{\binom{\x^k-\xbar}{t^k-h(\xbar)}}
	\leq{} &
		\varepsilon_k
		\left\|\binom{\x^k-\xbar}{h(\x^k)-h(\xbar)}\right\|
	\\
	\leq{} &
		\varepsilon_k
		\left\|\binom{\x^k-\xbar}{\kappa\|\x^k-\xbar\|}\right\|
	=
		\varepsilon_k\sqrt{1+\kappa^2}
		\|\x^k-\xbar\|,
	\end{align*}
	where the second inequality holds for \(k\) large enough.
	Arguing again by the arbitrariness of the sequence, we conclude that \(\vec{\bar v}\in\widehat\ncone_{\dom h}(\xbar)\).
	Finally, when \(h\) is calm relative to its domain at all points \(\x\in\dom h\) close to \(\xbar\), then the identity
	\(
		\widehat\ncone_{\dom h}(\x)\times\set{0}
	=
		\widehat\ncone_{\epi h}(\x,h(\x))
	\)
	holds for all such points, and a limiting argument then yields that
	\(
		\ncone_{\dom h}(\xbar)\times\set{0}
	=
		\ncone_{\epi h}(\xbar,h(\xbar))
	\)
	holds for the limiting normal cones.
	Therefore, the inclusion in \eqref{eq:NdomSubset} holds as equality, which concludes the proof.
\end{proof}

	When \(\cost\) is locally Lipschitz relative to its domain, the inclusion
	\(
		\bigl(
			-
			\trans{\jac\c(\xbar)}\lambar
			-
			\trans{\jac\ceq(\xbar)}\lameqbar
			,
			0
		\bigr)
	\in
		\partial\indicator_{\epi\cost}(\xbar,\cost(\xbar))
	\)
	derived in \eqref{eq:episubdiff} thus simplifies as
	\(
		-
		\trans{\jac\c(\xbar)}\lambar
		-
		\trans{\jac\ceq(\xbar)}\lameqbar
	\in
		\partial\indicator_{\dom\cost}(\xbar)
	\),
	and the conclusions of \cref{thm:outcomes} can be strengthened as follows.

	\begin{corollary}\label{thm:outcomesLip}%
		Additionally to \cref{ass:basic,ass:b}, suppose that \(\cost\) is locally Lipschitz continuous on its domain, assumed to be closed.
		Consider the iterates generated by \cref{alg:general} with \(\epsilon_{\rm p},\epsilon_{\rm d}>0\).
		Then, \(\inf_{k\in\N}\mu_k>0\) and exactly one of the following scenarios occurs:%
		\begin{enumerator}
		\item
			either the algorithm terminates returning an \eKKT[(\epsilon_{\rm p},\epsilon_{\rm d})] stationary point for \eqref{eq:P},
		\item
			or it runs indenfinitely with \(s_k\leq\epsilon_{\rm p}<p_k\) for all \(k\) large enough, and \(\seq{\alpha_k}\nearrow\infty\).
			Moreover, any accumulation point \(\xbar\) of \(\seq{\x^k}\) is KKT-stationary for the feasiblity problem
			\[
				\minimize_{\x\in\dom\cost}~\|[\c(\x)]_+\|_1+\|\ceq(\x)\|_1,
			\]
			in the sense that
			\(
				\zeros
			\in
				\partial\bigl[\|[\c(\xbar)]_+\|_1+\|\ceq(\xbar)\|_1\bigr]
				+
				\ncone_{\dom\cost}(\xbar)
			\).
		\end{enumerator}
	\end{corollary}

	Even under smoothness assumptions on \(\cost\), the outcomes presented in \cref{thm:outcomesLip} are usual for nonlinear programming solvers in the sense that, in general, one cannot exclude the possibility that the iterates remain trapped in an infeasible region.
	If \cref{alg:general} achieves approximate feasibility, then outcome (A) guarantees that it returns an approximate KKT stationary point for \eqref{eq:P}.
	Otherwise, outcome (B) specifies that the constraint violation is increasingly emphasized and \cref{alg:general} steers toward the solution of a feasibility problem.
	Several analogous results appear in the literature, such as \cite[Thm. 4.4]{conn1991globally}, \cite[Thm. 3.3]{demarchi2023constrained} and \cite[Thm. 6.2]{birgin2014practical}, among others.
	Stronger convergence guarantees can be established based on more restrictive assumptions,
	which are not of interest here beyond this remark.

	In practice, well-established solvers such as \ipopt{} \cite{waechter2006implementation} rely on auxiliary procedures whose purpose is to compute a new iterate that yields sufficient improvement.
	When in difficulty, \ipopt{} switches to a \emph{feasibility restoration phase} that attempts to decrease the constraint violation.
	Equipped with this mechanism, all the limit points it generates are feasible under the (tautological) assumption that ``the feasibility restoration phase $\ldots$ always terminates successfully'' \cite[Ass. G and Thm. 1]{waechter2005line}.

	Another strategy is that of \cite[\S 3.2]{curtis2012penalty}, whose update rules of penalty (and barrier) parameters are designed to result in rapid convergence to optimal points of the original problem \eqref{eq:P} or, at least, of the feasibility problem \eqref{eq:minInfeas}---in the context of nonlinear programming.
	Again, this technique promotes feasibility but additional properties are needed to obtain convergence guarantees.
	Our update rules at \cref{state:general:pk>eps,state:general:update_mu} draw inspiration from the Fiacco--McCormick approach \cite{fiacco1964sequential} and are comparable to the conservative strategy in \cite{curtis2012penalty}.

	Finally, also the question of whether the sequence of penalty parameters \(\seq{\alpha_k}\) remains bounded or not has practical interest, but theoretical findings in this direction are limited and circumscribed, even in the nonlinear programming setting.
	Boundedness of the sequence of multipliers, strict complementarity, and constraint qualifications are some of the assumptions adopted to establish the boundedness of penalty parameters; cf. \cite[\S 7]{birgin2014practical} and \cite[Ass. 3 and Thm. 1]{hallak2023adaptive}.

		\subsection{Barrier's properties}\label{sec:barrierProp}%
			According to its update rule in \cref{alg:general}, before a desired feasibility violation \(p_k\leq\epsilon_{\rm p}\) has been reached, \(\alpha_{k+1}=\alpha_k\) means that \(p_k\leq 2(m+\meq)\frac{-\b*(\rho*_k)}{\rho*_k}\), where \(\rho*_k=\nicefrac{\alpha_k}{\mu_k}\).
As shown in \cref{thm:rho}, regardless of whether \(\alpha_k\) is updated or not, \(\rho*_k\) grows linearly over the iterations, specifically as \(\rho*_k\geq\rho*_0\delta_{\rho*}^k\).
Therefore, having \(\alpha_{k+1}=\alpha_k\) implies in particular that either the constraint violation \(p_k\) is within a desired tolerance \(\epsilon_{\rm p}\), or that it is controlled from above by
\(
	2(m+\meq)
	\frac{-\b*(\rho*_k)}{\rho*_k}
\leq
	2(m+\meq)
	\frac{-\b*(\delta_{\rho*}^k\rho*_0)}{\delta_{\rho*}^k\rho*_0}
\),
where the inequality follows from (asymptotic) monotonicity of \(\nicefrac{-\b*(t^*)}{t^*}\), cf. \cref{thm:b*/t*}.
This means that a desired bound on feasibility violation whenever \(\alpha_k\) is not updated can be enforced through suitable choices of the barrier \(\b\).
This will be particularly significant in the convex case, for it can be shown that \(\alpha_k\) does eventually remain constant under reasonable assumptions.

\begin{lemma}\label{thm:pklinear}%
	Let \cref{ass:basic,ass:b} hold, and consider the iterates generated by \cref{alg:general}.
	Suppose that there exists \(\theta\in(0,1)\) such that the barrier \(\b\) satisfies \(\b(\theta t)\leq\theta\delta_{\rho*}\b(t)\) for every \(t<0\) (resp. for every \(t<0\) close enough to 0), where \(\delta_{\rho*}\coloneqq\min\set{\delta_\mu^{-1},\delta_\alpha}>1\).
	Then,
	\begin{equation}\label{eq:pklinear}
		\alpha_{k+1}=\alpha_k
	\quad\Rightarrow\quad
		p_k
	\leq
		\max\set{
			\epsilon_{\rm p},
			-2(m+\meq)\tfrac{\mu_0}{\alpha_0}\b*\bigl(\tfrac{\alpha_0}{\mu_0}\bigr)
			\theta^k
		}
	\end{equation}
	holds for every \(k\) (resp. for every \(k\) large enough).
\end{lemma}
\begin{proof}
	To simplify the presentation, without loss of generality let us set \(\epsilon_{\rm p}=0\).
	We have already argued that \(\alpha_{k+1}=\alpha_k\) implies \(p_k\leq 2(m+\meq)\pi_k\), where \(\pi_k\coloneqq\frac{-\b*(\delta_{\rho*}^k\rho*_0)}{\delta_{\rho*}^k\rho*_0}\) for all \(k\in\N\).
	It thus suffices to show that \(\pi_k\leq-\tfrac{\mu_0}{\alpha_0}\b*\bigl(\tfrac{\alpha_0}{\mu_0}\bigr)\theta^k\).
	To this end, notice that for every \(t^*>0\) one has
	\[
		\frac{-\b*(\delta_{\rho*}t^*)}{\delta_{\rho*}t^*}
	\leq
		\theta
		\frac{-\b*(t^*)}{t^*}
	\quad\Leftrightarrow\quad
		\b*(t^*)
	\leq
		\frac{\b*(\delta_{\rho*}t^*)}{\theta\delta_{\rho*}}
	=
		\sup_\tau\set{t^*\tau-\tfrac{\b(\theta\tau)}{\delta_{\rho*}\theta}}
	=
		\conj{\left(
			\tfrac{\b(\theta\cdot{})}{\theta\delta_{\rho*}}
		\right)}(t^*),
	\]
	hence, since \(\b*(t^*)=\infty\) for \(t^*<0\),
	\[
		\frac{-\b*(\delta_{\rho*}t^*)}{\delta_{\rho*}t^*}
	\leq
		\theta
		\frac{-\b*(t^*)}{t^*}
	\quad
		\forall t^*\in\R
	\quad\Leftrightarrow\quad
		\b(t)
	\geq
		\tfrac{\b(\theta t)}{\theta\delta_{\rho*}}
	\quad
		\forall t\in\R,
	\]
	which amounts to the condition in the statement.
	Under such condition, then, \(\pi_{k+1}\leq\theta\pi_k\) holds for every \(k\), leading to
	\(
		\pi_k
	\leq
		\pi_0\theta^k
	=
		\frac{-\b*(\rho*_0)}{\rho*_0}
		\theta^k
	\)
	as claimed.
\end{proof}

Though it would be tempting to seek barriers for which \(\seq{\pi_k}\) as in the proof vanishes at any desired rate, it can be easily verified that no choice of \(\b\) or \(\delta_{\rho*}\) can result in
\(
	\seq{\pi_k}
\)
converging any faster than linearly.
In fact,
\[
	\pi_{k+1}
=
	\frac{-\b*(\rho*_0\delta_{\rho*}^{k+1})}{\rho*_0\delta_{\rho*}^{k+1}}
>
	\frac{-\b*(\rho*_0\delta_{\rho*}^k)}{\rho*_0\delta_{\rho*}^{k+1}}
=
	\frac{1}{\delta_{\rho*}}\pi_k,
\]
where the inequality follows from monotonicity of \(-\b*\), cf. \cref{thm:b*}.
This shows that a linear decrease by a factor \(\delta_{\rho*}^{-1}\) is the fastest worst-case rate this lemma can guarantee, and that this can only happen in the limit.
\Cref{thm:pklinear} nevertheless identifies a property that allows us to judge the fitness of a barrier \(\b\) within the framework of \cref{alg:general}.
As we will see in \cref{sec:convex}, this will be particularly evident in the convex case, for it can be guaranteed that, under assumptions, \(\alpha_k\) eventually does remain constant, so that employing a barrier that complies with this requirement is a guarantee that eventually the infeasibility \(p_k\) of the iterates generated by \cref{alg:general} vanishes at R-linear rate.
This motivates the following definition.

\begin{definition}[behavior profiles of \(\b\)]\label{defin:kb}%
	We say that a barrier \(\b\) complying with \cref{ass:b} is \emph{asymptotically well behaved} if
	\begin{align*}
		\forall\theta\in(0,1)
	\quad
		\fillwidthof[c]{
			\kappa^{\rm max}_{\b}(\theta)
		}{
			\kappa_{\b}(\theta)
		}
	\coloneqq{} &
		\limsup_{t\to0^-}\frac{\b(\theta t)}{\theta\b(t)}
	<
		\infty
	&&
		\hspace*{-2cm}\text{and}\quad
		\lim_{\theta\to1^-}\kappa_{\b}(\theta)=1.
	\shortintertext{%
		If this condition can be strengthened to
	}
		\forall\theta\in(0,1)
	\quad
		\kappa^{\rm max}_{\b}(\theta)
	\coloneqq{} &
		\sup_{t<0}\frac{\b(\theta t)}{\theta\b(t)}
	<
		\infty
	&&
		\hspace*{-2cm}\text{and}\quad
		\lim_{\theta\to1^-}\kappa_{\b}^{\rm max}(\theta)=1,
	\end{align*}
	then we say that \(\b\) is \emph{well behaved} (not merely asymptotically).
	We call the functions \(\func{\kappa^{\rm max}_{\b},\kappa_{\b}}{(0,1)}{(1,\infty)}\) the \emph{behavior profile} and the \emph{asymptotic behavior profile} of \(\b\), respectively.
\end{definition}

In penalty-type methods, the update of a penalty parameter is typically decided based on the violation of the corresponding constraints.
Under the assumption that the barrier \(\b\) is (asymptotically) well behaved, \cref{thm:pklinear} demonstrates that in \cref{alg:general} (eventually) the condition \(\alpha_{k+1}=\alpha_k\) furnishes a guarantee of linear decrease of the infeasibility.
Insisting on continuity of \(\kappa_{\b}\) and \(\kappa_{\b}^{\rm max}\) at \(\theta=1\) in \cref{defin:kb} is a minor technicality ensuring that, regardless of the value of \(\delta_{\mu}\in(0,1)\) and \(\delta_\alpha>1\), for any (asymptotically) well behaved barrier there always exists \(\theta\in(0,1)\) such that \(\b(\theta t)\leq\theta\delta_{\rho*}\b(t)\) holds for every \(t<0\) (close enough to zero) as required in \cref{thm:pklinear}.
The result can thus be restated as follows.

\begin{corollary}\label{thm:linearsk}%
	Additionally to \cref{ass:basic,ass:b}, suppose that the barrier \(\b\) is (asymptotically) well behaved.
	Then, there exists \(\theta\in(0,1)\) such that the iterates of \cref{alg:general} satisfy \eqref{eq:pklinear} for all \(k\in\N\) (large enough).
\end{corollary}

When it comes to comparing different barriers, lower values of \(\kappa_{\b}\) are clearly preferable.
Notice that both \(\kappa_{\b}^{\rm max}\) and \(\kappa_{\b}\) are scaling invariant:
\[
	\kappa_{\beta\b}=\kappa_{\b(\beta{}\cdot{})}=\kappa_{\b}
\quad\text{and}\quad
	\kappa_{\beta\b}^{\rm max}=\kappa_{\b(\beta{}\cdot{})}^{\rm max}=\kappa_{\b}^{\rm max}
\qquad
	\forall\beta>0.
\]
Moreover, since
\[
	\kappa_{\b}(\theta)\geq\tfrac{1}{\theta}
\quad
	\forall\theta\in(0,1)
\]
(owing to monotonicity of \(\b\) and the fact that consequently \(\b(\theta t)\geq\b(t)\) for \(t<0\)), barriers attaining \(\kappa_{\b}(\theta)=\frac{1}{\theta}\) can be considered asymptotically optimal.
\Cref{table:kappa} shows that logarithmic barriers can attain such lower bound.

\begin{table}
	\[
	\setlength{\arraycolsep}{10pt}
		\begin{array}{|ccc|l}
			\multicolumn{1}{c}{\b(t) \text{ (for \(t<0\))}} & \multicolumn{1}{c}{\kappa_{\b}(\theta)} & \multicolumn{1}{c}{\kappa^{\rm max}_{\b}(\theta)}
		\\\cline{1-3}
			\vphantom{X^{\big|}}
			\frac{1}{p}(-t)^{-p}
			& \left(\frac{1}{\theta}\right)^{1+p}
			& \left(\frac{1}{\theta}\right)^{1+p}
			& \text{\rlap{(\(p>0\))}}
		\\
			\ln\left(1-\frac{1}{t}\right)
			& \frac{1}{\theta}
			& \left(\frac{1}{\theta}\right)^2
		\\
			-\ln(-t)
			& \frac{1}{\theta}
			& \infty
		\\
			\exp\left(-\frac{1}{t}\right)
			& \infty
			& \infty
		\\\cline{1-3}
		\end{array}
	\]
	\caption{%
		Examples of barriers  and their behavior profiles \(\kappa_{\b}\).
		A low \(\kappa_{\b}\) is symptomatic of good aptitude of \(\b\) as barrier within \cref{alg:general}.
		Geometrically, it indicates that \(\b\) well approximates the nonsmooth indicator \(\indicator_{\R_-}\).
		Functions like \(\exp\left(-\frac{1}{t}\right)\) growing too fast are unsuited, whereas logarithmic barriers attain an optimal asymptotic behavior profile \(\kappa_{\b}(\theta)=\frac{1}{\theta}\).
		The log-like barrier \(\b(t)=\ln\left(1-\frac{1}{t}\right)\) is well-behaved, whereas the logarithmic barrier \(\b(t)=-\ln\left(-t\right)\) is so only asymptotically.
	}%
	\label{table:kappa}%
\end{table}

		\subsection{The convex case}\label{sec:convex}%
			In this section we investigate the behavior of \cref{alg:general} when applied to convex problems.
In particular, we detail an asymptotic analysis in which the termination tolerances are set to zero, so that the algorithm (may) run indefinitely.
We demonstrate that under standard assumptions the iterates subsequentially converge to (global) solutions, and that the \(L^1\) penalty parameter \(\alpha\) is eventually never updated.

\begin{theorem}\label{thm:cvx}%
	Additionally to \cref{ass:basic,ass:b}, suppose that \(\inf\b>0\) and that problem \eqref{eq:P} is convex, namely that \(\cost\) and \(\c_i\), \(i=1,\dots,m\), are convex functions and that \(\ceq\) is affine.
	If there exists an optimal \KKT-triplet \((\x^\star,\y^\star,\yeq^\star)\) for \eqref{eq:P}, then the following hold for the iterates generated by \cref{alg:general} with \(\epsilon_{\rm p}=\epsilon_{\rm d}=0\):
	\begin{enumerate}
	\item \label{thm:cvx:omega}%
		Any accumulation point of the sequence \(\seq{\x^k}\) is a solution of \eqref{eq:P}.
	\item \label{thm:cvx:alphaConstant}%
		If, additionally, \(\seq{\x^k}\) remains bounded (as is the case when \(\dom\cost\) is bounded), \(\alpha_k\) is eventually never updated.
	\item \label{thm:cvx:b}%
		Further assuming that the barrier \(\b\) is asymptotically well behaved, so that there exists \(\theta\in(0,1)\) such that \(\b(\theta t)\leq\theta\min\set{\delta_\mu^{-1},\delta_\alpha}\b(t)\) for every \(t<0\) close enough to 0, then the feasibility violation eventually vanishes with rate
		\(
			p_k
		\leq
			2(m+\meq)\frac{-\b*(\nicefrac{\alpha_0}{\mu_0})}{\nicefrac{\alpha_0}{\mu_0}}
			\theta^k
		\).
	\end{enumerate}
\end{theorem}
\begin{proof}
	It follows from \cref{thm:KKT=>KKTa} that \(\x^\star\) solves \eqref{eq:Pa} for \(\alpha\coloneqq\max\set{\|\y^\star\|_\infty, \|\yeq^\star\|_\infty}\).
	For every \(k\), there exists
	\begin{equation}\label{eq:etak}
		\vec{\eta}^k
	\in
		\partial\cost_{\alpha_k,\mu_k}(\x^k)
	\quad\text{with}\quad
		\|\vec{\eta}^k\|\leq\varepsilon_k,
	\end{equation}
	where for \(\alpha,\mu>0\) we let
	\begin{equation}\label{eq:qam}
		\cost_{\alpha,\mu}
	\coloneqq
		\cost
		+
		\mu\Psi\circ\c
		+
		\mu\Psi^{\rm eq}\circ\ceq
	\geq
		\cost
		+
		\alpha\|[\c({}\cdot{})]_+\|_1
		+
		\alpha\|\ceq({}\cdot{})\|_1.
	\end{equation}
	Function \(\cost_{\alpha,\mu}\) is convex, because so are \(\Psi\circ\c\) and \(\Psi^{\rm eq}\circ\ceq\) by \cref{thm:psi,thm:psieq}, and satisfies the inequality as in \eqref{eq:qam} owing to \cref{thm:rho-infty}, having assumed that \(\b>0\).
	Notice further that
	\begin{align*}
		\cost_{\alpha,\mu}(\x^\star)
	={} &
		\cost(\x^\star)
		+
		\mu\Psi(\c(\x^\star))
		+
		\mu\Psi^{\rm eq}(\zeros)
	\\
	\leq{} &
		\cost(\x^\star)
		+
		\mu\Psi(\zeros)
		+
		\mu\Psi^{\rm eq}(\zeros )
	\\
	={} &
		\cost(\x^\star)
		+
		m\mu\psi(0)
		+
		\meq\mu\psi^{\rm eq}(0)
	\\
	={} &
		\cost(\x^\star)
		+
		m\mu\psi(0)
		+
		2\meq\mu\psi_{\nicefrac{\alpha}{2\mu}}(0)
	\\
		\dueto{(by \cref{thm:b*/t*})}
		\quad
	\leq{} &
		\cost(\x^\star)
		+
		(m+\meq)\mu\psi(0)
	\\
		\dueto{(by \eqref{eq:psieq})}
		\quad
	={} &
		\cost(\x^\star)
		+
		(m+\meq)\alpha\frac{-\b*(\nicefrac{\alpha}{\mu})}{\nicefrac{\alpha}{\mu}},
	\numberthis\label{eq:qam*}
	\end{align*}
	where the first inequality follows from (elementwise) monotonicity of \(\Psi\), and the third identity uses the fact that
	\(
		\psi_{\rho*}^{\rm eq}(0)=2\psi_{\rho*/2}(0)
	\)
	which is apparent from \eqref{eq:psidef} and \eqref{eq:psieqdef}.
	Next observe that, since \(\x^\star\) solves \eqref{eq:Pa} and is feasible, one has
	\begin{align*}
		\cost(\x^\star)
	={} &
		\cost(\x^\star)
		+
		\alpha\|[\c(\x^\star)]_+\|_1
		+
		\alpha\|\ceq(\x^\star)\|_1
	\\
	\leq{} &
		\cost(\x^k)
		+
		\bigl(
			\alpha_k-(\alpha_k-\alpha)
		\bigr)
		\underbracket*[0.5pt]{
			\bigl(
				\|[\c(\x^k)]_+\|_1
				+
				\|\ceq(\x^k)\|_1
			\bigr)
		}_{\eqqcolon\tilde p_k}
	\\
		\dueto{(by \eqref{eq:qam})}
		\quad
	\leq{} &
		\cost_{\alpha_k,\mu_k}(\x^k)
		-
		(\alpha_k-\alpha)\tilde p_k
	\\
		\dueto{(by \eqref{eq:etak})}
		\quad
	\leq{} &
		\cost_{\alpha_k,\mu_k}(\x^\star)
		+
		\innprod{\vec{\eta}^k}{\x^\star-\x^k}
		-
		(\alpha_k-\alpha)\tilde p_k
	\\
		\dueto{(by \eqref{eq:qam*})}
		\quad
	\leq{} &
		\cost(\x^\star)
		+
		(m+\meq)\alpha_k\frac{-\b*(\nicefrac{\alpha_k}{\mu_k})}{\nicefrac{\alpha_k}{\mu_k}}
		+
		\varepsilon_k\|\x^\star-\x^k\|
		-
		(\alpha_k-\alpha)\tilde p_k.
	\numberthis\label{eq:qamk}
	\end{align*}
	We next prove the assertions one by one.
	\begin{proofitemize}
	\item \ref{thm:cvx:omega}~
		If \(\seq{\alpha_k}\) is asymptotically constant, then according to the update rule in \cref{alg:general} one has that
		\(
			p_k
		\leq
			2(m+\meq)
			\frac{-\b*(\nicefrac{\alpha_k}{\mu_k})}{\nicefrac{\alpha_k}{\mu_k}}
		\)
		eventually always holds, and thus vanishes as \(k\to\infty\), see \cref{thm:b*/t*}.
		Any limit point \(\xbar\) of \(\seq{\x^k}\) is thus feasible, and furthermore belongs to \(\dom\cost\) since \(q(\x^k)\) remains bounded as is evident from the inequalities in \eqref{eq:qamk}.
		In this case, we conclude from \cref{thm:feas=>AKKT} that \(\x^\star\) is \AKKT-optimal, and thus optimal because of convexity.

		Suppose instead that \(\seq{\alpha_k}\nearrow\infty\), and by removing early iterates let us assume without loss of generality that \(\alpha_k>\alpha\) holds for any \(k\).
		Denoting \(\rho*_k\coloneqq\nicefrac{\alpha_k}{\mu_k}\nearrow\infty\), the inequality in \eqref{eq:qamk} yields that
		\begin{equation}\label{eq:pkleq}
			p_k
		\leq
			\tilde p_k
		\leq
			\frac{\alpha_k}{\alpha_k-\alpha}
			\Bigl(
				(m+\meq)
				\overbracket[0.5pt]{
					\frac{-\b*(\rho*_k)}{\rho*_k}
				}^{\to0}
				+
				\overbracket[0.5pt]{
					\frac{\varepsilon_k}{\alpha_k}
					\vphantom{\frac{-\b*(\rho*_k)}{\rho*_k}}
				}^{\to0}
				\|\x^\star-\x^k\|
			\Bigr),
		\end{equation}
		where the fact that \(\tfrac{-\b*(\rho*_k)}{\rho*_k}\to0\) follows from \cref{thm:b*/t*}.
		Along any convergent subsequence, it is clear that \(p_k\to0\) and that \(\cost(\x^k)\) remains bounded, and again we conclude that the limit point is optimal for \eqref{eq:P}.

	\item \ref{thm:cvx:alphaConstant}~
		To arrive to a contradiction, suppose that \(\alpha_k\nearrow\infty\), and again assume without loss of generality that \(\alpha_k>\alpha\) holds for all \(k\).
		It then follows from \eqref{eq:pkleq} that
		\[
			\frac{
				p_k
			}{
				2(m+\meq)
				\frac{-\b*(\rho*_k)}{\rho*_k}
			}
		\leq
			\frac{1}{2}
			\frac{\alpha_k}{\alpha_k-\alpha}
			\Bigl(
				1
				+
				\frac{\varepsilon_k}{\alpha_k}
				\frac{\rho*_k}{-\b*(\rho*_k)}
				R
			\Bigr)
		\to
			\frac12
		\quad
			\text{as }k\to\infty,
		\]
		and in particular \(p_k\) is eventually always smaller than
		\(
			2(m+\meq)
			\tfrac{-\b*(\rho*_k)}{\rho*_k}
		\).
		According to the update rule of \(\alpha_k\) in \cref{alg:general}, this means that \(\alpha_k\) is eventually constant, which is a contradiction.

	\item \ref{thm:cvx:b}~
		Follows from \cref{thm:pklinear}.
	\qedhere
	\end{proofitemize}
\end{proof}

The \pipa{} algorithm of \cite{chouzenoux2020proximal}, specialized to the convex setting, generates primal-dual sequences that remain bounded,
offering a theoretical advantage over \cref{thm:cvx}, where boundedness is assumed.
The difficulty in recovering the results of \cite{chouzenoux2020proximal} appears to stem from the constraint relaxation occurring in \eqref{eq:Pa}, which on the other hand is the key for \pippo{} to
handle equality constraints and infeasible starting points.

	\section{Numerical experiments}\label{sec:numerical}%
		This section is dedicated to experimental results and comparisons with other numerical approaches for constrained structured optimization.
The modular structure of the proposed framework allows us to combine \pippo{} with a variety of penalty-barrier envelopes and inner solvers (insofar as they provide suitable guarantees).
The perfomance and behavior of \pippo{} is illustrated in different variants, considering three barrier functions, namely \(\b(t)=-\frac1t\), \(\b(t)=\ln(1-\frac{1}{t})\), and \(\b(t)=-\ln(-t)\) (all extended as $\infty$ on \(\R_+\)) denominated \emph{inverse}, \emph{log-like}, and \emph{log}, respectively.
The numerical comparison will comprise two data science tasks and two \emph{ad hoc} illustrative problems.
These numerical tests will also highlight the influence of the barrier function on the performance of \pippo{}, supporting the quality assessment of \cref{sec:barrierProp}.

The performance of \pippo{} is compared against those of \ipprox{} \cite[Alg. 1]{demarchi2024interior},{}
\alps{} \cite[Alg. 4.1]{demarchi2024implicit}, and the well-known \ipopt{} \cite{waechter2006implementation}.
\ipprox{} builds upon a pure interior point scheme and solves the barrier subproblems with a tailored adaptive proximal-gradient algorithm, extending the strategy of \pipa{} \cite{chouzenoux2020proximal} to the nonconvex setting.
\alps{} belongs to the family of augmented Lagrangian algorithms and does not require a custom subsolver---suitable subsolvers for \pippo{} can be applied within \alps{} and viceversa.
	The closely related solvers {\algnamefont OpEn} \cite{sopasakis2020open} and {\algnamefont Alpaqa} \cite{pas2022alpaqa} also build on the augmented Lagrangian framework and provide easy-to-use high-performance implementations; however, these focus on nonlinear programming problems and cannot support generic prox-friendly cost functions as \alps{} does.
	Finally, \ipopt{} is a software package for large-scale sparse nonlinear optimization; it implements an interior point line search filter method.
	\ipopt{} can address problems of the form \eqref{eq:P} where the input functions, particularly $\cost$, should be at least continuously differentiable.

	\pippo{}'s minimal assumptions on the subsolver leave us much freedom in its selection.
	To handle potential (structured) nonsmoothness of the cost \(\cost\), we chose \panocp{} \cite{demarchi2022proximal,stella2017simple}, a proximal-gradient-based solver that can exploit directions of quasi-Newton type while ensuring convergence with a backtracking linesearch, see also \cite[\S 5.1]{themelis2018forward}.
	We also tested the nonmonotone proximal-gradient method with adaptive stepsizes of \cite{demarchi2023monotony},
	but in our comparisons we only retained the results with \panocp{}, set as default solver in our implementation, as it consistently demonstrated better performance.

Patterning the simulations of \cite[\S 5.2]{demarchi2024interior}, we first examine the nonnegative PCA problem in \cref{sec:Numerics:NonnegPCA} to evaluate \pippo{} in several variants and compare it against the other solvers.
Then, \cref{sec:Numerics:MatrixCompletion} focuses on a low-rank matrix completion task, a fully nonconvex problem with thousands of variables and constraints,
contrasting \pippo{} to \alps{} and \ipopt.
Finally, the exact penalty behavior and the ability to handle hidden equalities are illustrated and discussed in \cref{sec:Numerics:ExactPenalty,sec:Numerics:Equalities}, respectively.
Additional observations that support the analysis of barriers' properties are included in \cref{sec:Numerics:BarrierProp}.

The source code of our implementation has been archived for reproducibility of the numerical results presented in this paper;
it can be found on Zenodo at
\textsc{doi}: \href{https://doi.org/\TheCodeZenodoDOI}{\TheCodeZenodoDOI}.

		\subsection{Implementation details}
			We describe here details pertinent to our implementation \pippo{} of \cref{alg:general}, such as the initialization and update of algorithmic parameters.
These numerical features tend to improve the practical performances, without compromising the convergence guarantees.
\ipprox{} is available from \cite{demarchi2024interior} and adopted as is, whereas \alps{} is a slight modification of the code from \cite{demarchi2024implicit} to be comparable with \pippo{}, as detailed below.

Our implementation of \pippo{} accepts problems formulated in the form
\[
	\minimize_{\x\in\R^n}{}~
	f(\x)+g(\x)
	\quad
	\stt{}~
	\vec{l} \leq \c(\x) \leq \vec{u}
	,
\]
with bounds defined by extended-real-valued vectors $\vec{l}$ and $\vec{u}$.
In a preprocessing phase, these vectors are parsed to reformulate the problem data in the format \eqref{eq:P} and to instantiate the appropriate penalty-barrier functions to treat inequality and equality constraints as described above.

Default parameters for \pippo{} are
$\mu_0 = 1$ and
$\delta_\varepsilon = \delta_\mu = \nicefrac14$ as in \ipprox{},
$\delta_\alpha = 2$ as in \alps{},
and $\alpha_0 = 1$.
The initial tolerance $\varepsilon_0$ for \pippo{} (and \alps{}) is chosen adaptively, based on the user-provided starting point ${\x}^0$ and penalty-barrier parameters.
Following the mechanism implemented in \ipprox{},
we set $\varepsilon_0 = \max\set{\epsilon_{\rm d}, \varepsilon_{\min}, \min\set{\kappa_\varepsilon \eta_0, \varepsilon_{\max} } }$, where $\kappa_\varepsilon \in (0, 1)$ and $\varepsilon_{\max} \geq \varepsilon_{\min} \geq 0$ are user-specified parameters (default $\kappa_\varepsilon = 10^{-2}$, $\varepsilon_{\max} = 1$, $\varepsilon_{\min} = 10^{-6}$)
and $\eta_0$ is an estimate of the initial stationarity measure,
as evaluated by (executing one iteration of) the inner solver invoked at $({\x}^0, \alpha_0, \mu_0)$.
The barrier parameter is updated according to the rule presented in \cref{rem:variant}.
For simplicity, no infeasibility detection mechanism nor artificial bounds on penalty and barrier parameters have been included.

We run \alps{} with the same settings as in \cite{demarchi2023constrained,demarchi2024implicit} apart from the following features to match \pippo{}:
the initial penalty parameter is fixed ($\alpha_0=1$) and not adaptive,
the tolerance reduction factor is set to $\delta_\varepsilon=\nicefrac{1}{4}$ instead of $\delta_\varepsilon=\nicefrac{1}{10}$,
and the initial inner tolerance is selected adaptively.
We always initialize \alps{} with dual estimate \(\y^0=\zeros\).
The \panocp{} subsolver is considered with its default tuning, namely{}
L-BFGS directions (memory 5) and monotone linesearch strategy as in \cite{demarchi2022proximal}.
We use \ipopt{} version 3.13.3, called via \mexipopt{} \cite{bertolazzi2023mexipopt}, with linear solver \mumps{}.
The only non-default optional values we set are the \texttt{limited-memory} Hessian approximation and the desired termination tolerance \texttt{tol}, as specified below.
For consistency with \ipopt{}'s termination condition \cite[\S 2.1]{waechter2006implementation}, \pippo{} requires approximate stationarity measured in the \(L^\infty\)-norm, deviating from \cref{defin:eKKT}; this applies also to \alps{} and \ipprox{}.

For $P$ the set of problems and $S$ the set of solvers,
let $t_{s,p}$ denote the user-defined metric for the computational effort required by solver $s \in S$ to solve instance $p \in P$ (lower is better).
We will monitor the (total) number of gradient evaluations, so that the computational overhead triggered by backtracking is fairly accounted for, the number of (outer) iterations, and the wall-clock runtime.
Then, we display \emph{data profiles} to graphically summarize our numerical results and compare different solvers.
A data profile is the graph of the cumulative distribution function $\func{f_s}{[0,\infty)}{[0,1]}$ of the evaluation metric, namely
$f_s(t) \coloneqq |\set{p\in P}[t_{s,p}\leq t]| / |P|$.
As such, each data profile reports the fraction of problems $f_s(t)$ that can be solved (for a given tolerance $\epsilon$) by solver $s$ with a computational budget $t$, and therefore it is independent of the other solvers \cite{more2009benchmarking}.

		\subsection{Nonnegative PCA}\label{sec:Numerics:NonnegPCA}%
			Principal component analysis (PCA) aims at estimating the direction of maximal variability of a high-dimensional dataset.
Imposing nonnegativity of entries as prior knowledge, we address PCA restricted to the positive orthant:
\begin{equation}
	\label{eq:nonneg_pca}
	\maximize_{\x\in\R^n}~
	\x^\top \mat{Z} \x
	\quad
	\stt{}~
	\| \x \| = 1 ,~
	\x \geq \zeros .
\end{equation}
This task falls within the scope of \eqref{eq:P}, with $f(\x) \coloneqq - \x^\top \mat{Z} \x$, $g(\x) \coloneqq \indicator_{\|\cdot\| = 1}(\x)$, and $\c(\x) = - \x$,
and has been considered in \cite{demarchi2024interior} for validating \ipprox{} and tuning its hyperparameters.
Here, we start by showing that all \pippo{}'s variants significantly outperform \ipprox{} on the same instances the latter was fine-tuned with (using the inverse barrier); cf. \cite[\S 5.2]{demarchi2024interior}.
Then, we investigate how the computational effort of \pippo{}'s default setup (that is, with log-like barrier and subsolver \panocp{}) scales with the problem size and the accuracy requirement.

\paragraph*{Setup}
	We generate synthetic problem data as in \cite[\S 5.2]{demarchi2024interior}.
	For a problem size $n\in\N$,
	let $\mat{Z} = \sqrt{\sigma_n} \vec{z} \vec{z}^\top + \mat{N} \in\R^{n\times n}$, where $\mat{N} \in\R^{n\times n}$ is a random noise matrix,
	$\vec{z}\in\R^n$ is the true (random) principal direction,
	and $\sigma_n > 0$ is the signal-to-noise ratio.
	We consider some dimensions $n$ and,
	for each dimension,
	the set of problems parametrized by $\sigma_n \in \set{0.05, 0.1, 0.25, 0.5, 1.0}$ and $\sigma_s \in \set{0.1, 0.3, 0.7, 0.9}$, which control the noise and sparsity level, respectively.
	There are 5 choices for $\sigma_n$, 4 for $\sigma_s$, and, for each set of parameters, 2 instances are generated with different problem data $\mat{Z}$ and starting point ${\x}^0$.
	Overall, each solver-settings pair is invoked on 40 different instances for each dimension $n$.

	A strictly feasible starting point ${\x}^0$ is generated by sampling a uniform distribution over $[0,3]^n$ and projecting onto $\dom g = \set{\x\in\R^n}[\|\x\| = 1]$.
	This property is necessary for \ipprox{} but not for the other solvers.
	We will test \pippo{}, \ipopt{} and \alps{} also with arbitrary initialization, in which case ${\x}^0$ is generated by sampling a uniform distribution over $[-3,3]^n$ and then projecting onto $\dom g$.
		Furthermore, since \ipprox{} requires a nonnegative-valued barrier function \cite[Ass. 2]{demarchi2024interior}, its variant with log barrier is not included in these tests.
		Finally, \ipopt{} tackles \eqref{eq:nonneg_pca} encoded with the nonlinear equality constraint $\sum_{i=1}^n \x_i^2 = 1$ and simple lower bounds on $\x$.

\paragraph*{Barriers and subsolvers}
	\Cref{alg:general} is controlled by, and its performance depends on, several algorithmic hyperparameters, such as the (sequences of) barrier and penalty parameters, the choice of barrier function $\b$, and the subsolver adopted at \cref{state:general:xk}.
	We now focus on the effect of the barriers specified in \cref{table:psi},
	for different levels of accuracy requirements, testing all solver variants on \emph{tiny} and \emph{small} instances with problem dimensions $n \in \set{10, 15, 20}$ and $n\in\{100,150,200\}$, respectively.
	Moreover, because of the excessive runtime to perform all simulations for \ipprox{},
	we exclude it altogether for the high accuracy tests, for which we consider starting points \({\x}^0\) that are not necessarily (strictly) feasible.

The results are graphically summarized in \cref{fig:nonneg_pca_time} with data profiles relative to runtime.
All solvers returned successfully within the desired primal-dual tolerances.
Across all accuracy levels,
\pippo{} inverse operates consistently better than the other variants of \pippo{},
all of which outperform \ipprox{}.
In particular,
the overall effort (in terms of runtime) required by the inverse barrier (in both \pippo{} and \ipprox{}) is less than with the log-like and log barriers , whose profiles almost overlap.
 With increasing accuracy it becomes even more efficient to adopt \pippo{} over \ipprox{}, which performs gradually more poorly.
The slow tail convergence typical of first-order schemes badly affects the scalability of \ipprox{},
whereas the adoption of a quasi-Newton scheme within the subsolver \panocp{} of \pippo{} and \alps{} leads to fast local convergence in practice.
	Despite the ``optimality'' of logarithmic barriers ascertained in \cref{sec:barrierProp},
	the \emph{overall} computational effort (measured in terms of gradient evaluations or runtime) also depends on the subsolver's efficiency in solving the subproblems.

	For tiny instances or low accuracy, the overall runtime of \pippo{} tends to be comparable with that of \ipopt{}.
	For high accuracy requirements and manageable problem sizes, \ipopt{} can profit from the more computation-intense iterations, whereas the cheaper first-order steps of \pippo{}'s subsolver pay off with lower accuracies.
	Adopting the same first-order subsolver, \alps{} appears to finish ahead of \pippo{} and \ipopt{}.
	We can attribute this advantage to the simple structure of the explicit constraints \(\x\geq\zeros\) in \eqref{eq:nonneg_pca}, since the superiority of \alps{} vanishes with more difficult constraints, as witnessed by the following \cref{sec:Numerics:MatrixCompletion}.

\begin{figure}[tbh]
	\centering%
		\includetikz[width=\linewidth]{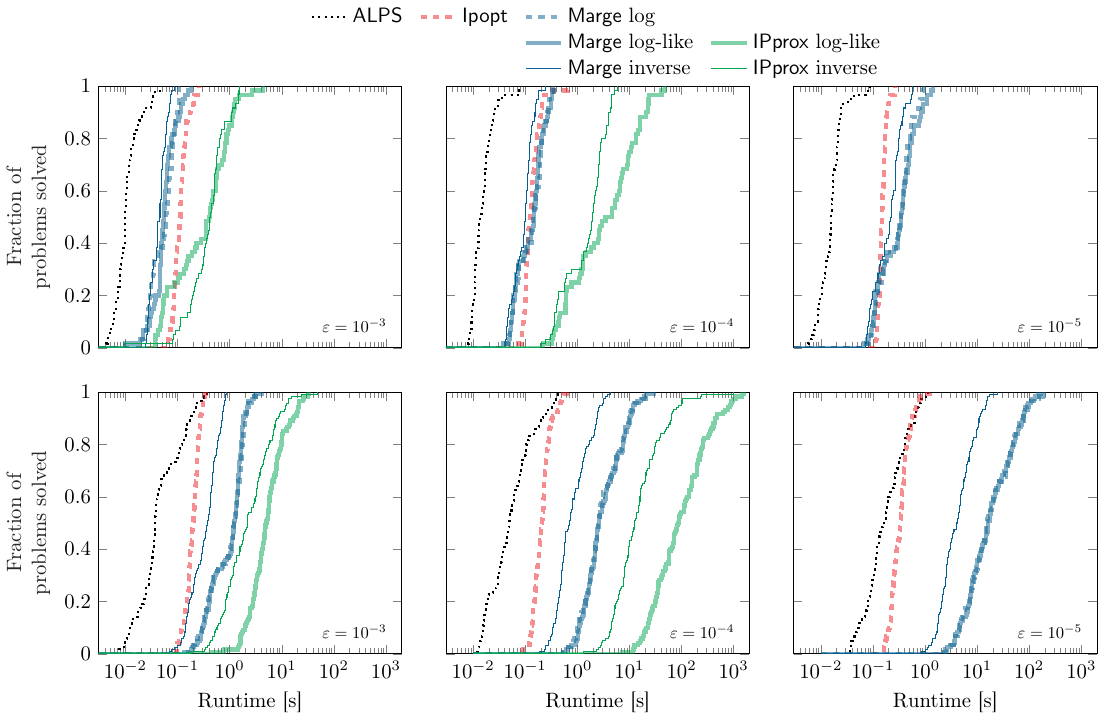}%
	\caption[]{%
		Nonnegative PCA problem \eqref{eq:nonneg_pca}: comparison of solvers on small (bottom) and large instances (top) with low, medium and high accuracy $\epsilon_{\rm p} = \epsilon_{\rm d} = \varepsilon\in \{10^{-3}, 10^{-4}, 10^{-5}\}$ (left to right) using data profiles relative to runtime.
		Results for \ipprox{} (green solid lines) are not included for the high accuracy tests due to excessive runtime; in these simulations, the other solvers are initialized with a possibly infeasible guess.
			In all plots, \pippo{}'s profiles for log-like (blue thick solid lines) and log (blue thick dashed lines) barriers almost coincide.
			}%
	\label{fig:nonneg_pca_time}%
\end{figure}

\paragraph*{Problem size and accuracy}
	To investigate scalability and influence of accuracy requirements,
	we consider larger instances of \eqref{eq:nonneg_pca} with dimensions $n \in \set{10, \lceil 10^{1.5} \rceil, 10^2, \lceil 10^{2.5} \rceil , 10^3}$ and tolerances $\epsilon_{\rm p} = \epsilon_{\rm d} = \varepsilon \in \set{10^{-3}, 10^{-4}, 10^{-5}}$,
	and invoke the default solver \pippo{} (with \panocp{} subsolver and log-like barrier) without time limit.
	For each of these tolerance parameters,
	we generate 2 instances (as described above) for each set of parameters, leading to a total of 200 problem instances to be solved for each accuracy level.

All instances are solved up to the desired primal-dual tolerances.
The influence of problem size and tolerance is depicted in \cref{fig:nonneg_pca_scaling}, which displays for each pair $(n,\varepsilon)$ the number of gradient evaluations and runtime with a jitter plot (for a better visualization of the distribution of numerical values over categories).
The empirical cumulative distribution function and the associated median value are also indicated.
This chart visualizes how problem size and accuracy requirement affect the solution process, and reveals the stark effect of both $n$ and $\varepsilon$.

For low accuracy, \pippo{} scales relatively well with the problem size, whereas larger problems become prohibitive for high accuracy.
This behavior is typical of first-order methods, due to their slow tail convergence,
and we take it as a motivation for investigating the interaction between subpoblems and subsolvers in future works.
Nevertheless, these experiments (and those forthcoming) demonstrate \pippo{}'s capability to handle thousands of variables and constraints in a fully nonconvex optimization landscape.
These results witness a tremendous improvement over \ipprox{}, not only in the practical performance but also in the flexibility of use,
as \pippo{} can be initialized at infeasible points and can take advantage of general-purpose fast subsolvers, such as \panocp{}.

\begin{figure}[tbh]
	\centering
	\includetikz{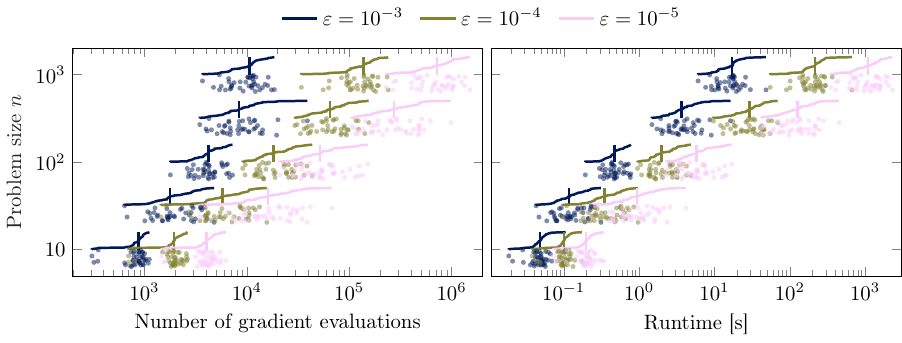}%
	\caption{Nonnegative PCA problem \eqref{eq:nonneg_pca} with \protect\pippo{} log-like:
		comparison for increasing accuracy requirements (decreasing tolerances $\epsilon_{\rm p} = \epsilon_{\rm d} = \varepsilon$) and problem sizes $n$, relative to gradient evaluations (left) and runtime (right).
		Combination of jitter plot (dots) and empirical cumulative distribution function (solid line) with median value (vertical segment).%
	}%
	\label{fig:nonneg_pca_scaling}
\end{figure}

		\subsection{Low-rank matrix completion}\label{sec:Numerics:MatrixCompletion}%
			Given an incomplete matrix of (uncertain) ratings $\mat{Y}$, a common task is to find a complete ratings matrix $\mat{X}$ that is a parsimonious representation of $\mat{Y}$, in the sense of low-rank, and such that $\mat{Y} \approx \mat{X}$ for the entries available \cite{marecek2017matrix}.
Let $\nusers$ and $\nmovies$ denote the number of users and items, respectively, and
let the rating $\mat{Y}_{i,j}$ by the $i$th user for the $j$th item range on a scale defined by constants $Y_{\min}$ and $Y_{\max}$.
Let $\Omega$ represent the index set of observed ratings, and $|\Omega|$ the cardinality of $\Omega$.
The ratings matrix $\mat{Y}$ could be very large and often most of the entries are unobserved, since a given user will only rate a small subset of items.
Low-rankness of $\mat{X}$ can be enforced by construction, with the Ansatz $\mat{X} \equiv \mat{U} \mat{V}^\top$, as in dictionary learning.
In practice,
for some prescribed embedding dimension $\natoms$,
we seek a user embedding matrix $\mat{U} \in \R^{\nusers\times\natoms}$ and an item embedding matrix $\mat{V}\in\R^{\nmovies\times\natoms}$.
Each row $\mat{U}_i$ of $\mat{U}$ is a $\natoms$-dimensional vector representing user $i$,
while each row $\mat{V}_j$ of $\mat{V}$ is a $\natoms$-dimensional vector representing item $j$.
We address the joint completion and factorization of the ratings matrix $\mat{Y}$,
encoded in the following form:
\begin{align*}
\numberthis\label{eq:matrixcompletion}
	\minimize_{\substack{\mat{U}\in\R^{\nusers\times\natoms},\\ \mat{V}\in\R^{\nmovies\times\natoms}}}
\quad &
	\frac{1}{|\Omega|}
	\sum_{(i,j)\in\Omega} \left( \innprod{\mat{U}_i}{\mat{V}_j} - \mat{Y}_{i,j} \right)^2
	+
	\frac{\lambda}{\nmovies} \sum_{j=1}^{\nmovies} \| \mat{V}_j \|_0
\\
	\stt
\quad &
	\max\set{Y_{\min},\mat{Y}_{i,j}-1}
	\leq
	\innprod{\mat{U}_i}{\mat{V}_j}
	\leq
	\min\set{Y_{\max},\mat{Y}_{i,j}+1}
&&
	\forall (i,j)\in\Omega,
\\
&
	Y_{\min}
	\leq
	\innprod{\mat{U}_i}{\mat{V}_j}
	\leq
	Y_{\max}
&&
	\forall (i,j) \notin\Omega,
\\
&
	\|\mat{U}_i\|_2 = 1
&&
	\forall i \in \set{1,\dots,\nusers}.
\end{align*}
While aiming at $\mat{U} \mat{V}^\top \approx \mat{Y}$, the model in \eqref{eq:matrixcompletion} sets the rating range \([Y_{\min},Y_{\max}]\) as a hard constraint for all predictions;
a tighter constraint is imposed to observed ratings.
Following \cite[\S 6.2]{themelis2018forward}, we explicitly constrain the norm of the dictionary atoms \(\mat{U}_{i}\), without loss of generality, to reduce the number of equivalent (up to scaling) solutions;
this norm specification is included as an indicator in the nonsmooth objective term \(g\).
Furthermore, we encourage sparsity of the coefficient representation \(\mat{V}_{j}\) with the $\|\cdot\|_0$ penalty, which counts the nonzero elements,
scaled with a regularization parameter $\lambda \geq 0$.
Overall, this problem has $n\coloneqq\natoms(\nusers+\nmovies)$ decision variables and $m\coloneqq 2\nusers\nmovies$ inequality constraints.
All terms ($f$, $g$, and $\c$) are nonconvex,
as well as the (unbounded) feasible set.

It appears nontrivial to find a strictly feasible point for \eqref{eq:matrixcompletion},
in the sense of \cite[Def. 2]{demarchi2024interior},
which is required for initializing \ipprox{},
thus highlighting a major advantage of \pippo{}.
	Another feature of \pippo{} is the versatility of model \eqref{eq:P}, which enables an effortless handling of the nonsmooth term $\|\cdot\|_0$, in stark contrast with \ipopt{}.
	Although \eqref{eq:matrixcompletion} can be reformulated as a nonlinear program via \eqref{eq:l0_as_linprog} using $3\nmovies\natoms$ auxiliary variables, $4\nmovies\natoms$ linear and $\nmovies$ nonlinear additional constraints, this extended formulation hinders the scalability to large datasets.
	For comparison, we will consider instances of \eqref{eq:matrixcompletion} with and without the sparsity-promoting term \(\|\cdot\|_0\) in the objective.

\paragraph*{Setup}
	We consider the \emph{MovieLens 100k} dataset,\footnote{The entire dataset is available at \url{https://grouplens.org/datasets/movielens/100k/}.}
	which contains $1000023$ ratings for $3706$ unique movies (the dataset contains some repetitions in movie ratings and we have ignored them);
	these recommendations were made by $6040$ users on a discrete rating scale from $Y_{\min}=1$ to $Y_{\max}=5$.
	If we construct a matrix of movie ratings by the users, then it is a sparse unstructured matrix with only $4.47\%$ of the total entries available.

	We compare \pippo{} to \alps{} and \ipopt{}, testing their scalability with instances of increasing size.
		The set of problems consists of \emph{small} and \emph{large} instances, as well as instances \emph{with} and \emph{without} the \(\|\cdot\|_0\) term in \eqref{eq:matrixcompletion}.
		We set the regularization parameter \(\lambda = 10^{-2}\), randomly generate 4 starting points for each problem instance, and invoke each solver with the primal-dual tolerances \(\epsilon_{\rm p}=\epsilon_{\rm d} = 10^{-3}\) and without time limit.
		First, we fix the number of atoms to $\natoms=5$ and consider the small instances corresponding to subsets of $\nusers \in \set{3,4,\ldots,7}$ users (always starting from the first one),
		for a total of 20 calls to each solver variant.
		For these problem instances the sizes range $n \in [1790, 3450]$ and $m \in [2130, 9562]$.
		Then, we fix the number of atoms to $\natoms=10$ and consider the large instances of \eqref{eq:matrixcompletion} corresponding to subsets of $\nusers \in \set{11,12,\ldots,20}$ users, for a total of 40 calls to each solver variant.
	For these problem instances the sizes range $n \in [7630, 9930]$ and $m \in [16544, 38920]$; \ipopt{} is not tested on these large instances with $\|\cdot\|_0$ because of excessive runtime.

\paragraph*{Results}
	A summary of the numerical results is depicted in \cref{fig:lowrank} as data profiles, while median values are reported in \cref{table:lowrank}.
	Except \ipopt{} on instances with $\|\cdot\|_0$, all solver variants were able to find a solution up to the specified primal-dual tolerance.
		For small instances, there is not a clear winner among \pippo{}'s variants and \alps{}, but \ipopt{} appears to fall behind.
		\ipopt{} returned unsuccessfully on 30\% of calls for small instances with $\|\cdot\|_0$, requiring more than ten times the runtime of other solvers; for this reason it was not tested on the large instances.
		Although \ipopt{} requires significantly fewer gradient evaluations, its overall runtime soars, even on the instances without $\|\cdot\|_0$.
		Indeed, the metrics reported in \cref{table:lowrank} indicate that, to solve at least half of all large instances without $\|\cdot\|_0$, \ipopt{}'s runtime is six times longer than \pippo{}'s with log-like barrier.
		On the large instances with \(\|\cdot\|_0\), \pippo{} with log-like and log barriers consistently outperformed the variant with inverse barrier and \alps{}, cutting the runtime almost in half.
 	These results show that \pippo{} can handle problems with thousands of variables and constraints, with a purely primal approach, and be as effective as well-established general-purpose solvers.

\begin{figure}[tbh]
	\centering%
		\includetikz[width=\linewidth]{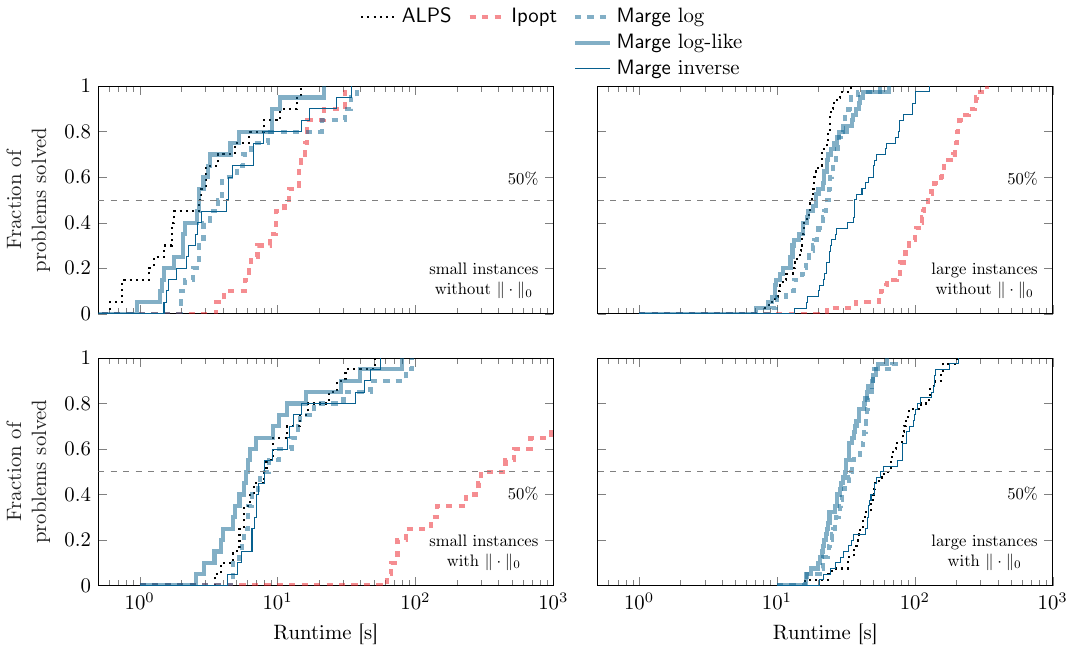}%
	\caption{%
		Matrix completion problem \eqref{eq:matrixcompletion}: comparison of different solvers and variants on small (left) and large instances (right), with (bottom) and without the \(\|\cdot\|_0\) regularization (top), using data profiles relative to runtime.
		\ipopt{} (red thick dashed lines) is not included on large instances with \(\|\cdot\|_0\) because of excessive runtime.	}%
	\label{fig:lowrank}
\end{figure}

	\begin{table}[tbh]
		\centering
		\begin{tabular}{cccccc}
			\hline
			& \multicolumn{2}{c}{Small instances} & \multicolumn{2}{c}{Large instances} & \\
			Solver variant & \# gradient eval. & Runtime [s] & \# gradient eval.  & Runtime [s] & \\
			\hline
			\pippo*{} log-like 	& 4103 	& 2.7 & 6874 & 18.8 & \parbox[t]{2.25ex}{\multirow{5}{*}{\rotatebox[origin=c]{-90}{without \(\|\cdot\|_0\)}}}\\
			\pippo*{} inverse 	& 7262 	& 4.2 & 15064 & 36.6 &\\
			\pippo*{} log 		& 5941 	& 3.7 & 8657 & 23.0 &\\
			\alps{} 			& 6223 	& 2.7 & 11959 & 17.4 &\\
			\ipopt{} 			& 161 	& 11.2 & 215 & 125.2 &\\
			\hline
			\pippo*{} log-like 	& 4248 	& 5.9 & 7156 & 31.0 & \parbox[t]{2.25ex}{\multirow{5}{*}{\rotatebox[origin=c]{-90}{with \(\|\cdot\|_0\)}}} \\
			\pippo*{} inverse 	& 6789 	& 7.6 & 15239 & 56.5 & \\
			\pippo*{} log 		& 5805 	& 7.5 & 7712 & 34.4 & \\
			\alps{} 			& 6906 	& 8.0 & 15239 & 59.8 & \\
			\ipopt{} 			& 1176 	& 296.7 & --- & --- &\\
			\hline
		\end{tabular}
		\caption{%
			Matrix completion problem \eqref{eq:matrixcompletion}: performance comparison of different solvers and variants, reporting the computational effort needed to achieve 50\% of problems solved.
			The reported values arise from the intersection of the data profiles in \cref{fig:lowrank} with the 50\% horizontal line; analogously for the number of gradient evaluations.
			\ipopt{} was not tested on large instances with $\|{}\cdot{}\|_0$ because of excessive runtime.%
		}%
		\label{table:lowrank}%
	\end{table}%

\bigskip

Among all instances of \eqref{eq:matrixcompletion} considered so far, we observed that the penalty parameter \(\alpha_k\) was rarely, if ever, updated by any variants of \pippo{}.
For illustrative purposes, we solved once again the large instance with $\nusers = 11$ from above, but starting with the much smaller penalty parameter \(\alpha_0 = 10^{-4}\).
The penalty behavior is displayed in \cref{fig:lowrank_alpha} (left panel), tracing the number of updates for \(\alpha_k\) along the iterations for each of the 4 starting points.
The solution process of each solver is consistent throughout all initializations, and the total number of penalty updates is also distinctive of \pippo{} (except for the log variant) and \alps{}; see right panel of \cref{fig:lowrank_alpha}.
The latter takes between 8 and 11 updates, whereas the former only 4 or 5 (7 updates for the log variant).
Such updates takes place in \alps{} whenever \emph{local} improvement in feasibility is deemed insufficient from one iteration to the next.
The same favorable behavior of \pippo{} is enabled by the relaxed condition at \cref{state:general:pk>eps} of \cref{alg:general},
which does not require any sufficient improvement at every iteration,
but instead monitors \emph{globally} how the constraint violation $p_k$ vanishes.
Correspondingly, only the barrier parameter $\mu_k$ is decreased in order to reduce the complementarity slackness $s_k$, see \cref{thm:CS}.
When active, this \emph{exact} penalty quality prevents the barrier to yield too much ill-conditioning.

 	While the update of the penalty parameter \(\alpha\) for log-like and inverse barriers follow similar profiles (though exhibiting a discrepancy coherent with \cref{sec:barrierProp}),
 	\cref{fig:lowrank_alpha} (left panel) portrays a sharper increase of the penalty parameter when adopting the logarithmic barrier.
 	Most directly, this different behavior emerges because, starting with \((\alpha_0,\mu_0)=(10^{-4},1)\), \(\b*(\nicefrac{\alpha_k}{\mu_k})=-1-\ln(\nicefrac{\alpha_k}{\mu_k})>0\) holds for some iterations until the ratio \(\nicefrac{\alpha_k}{\mu_k}\) becomes large enough.
 	In practice, during these initial iterations, \cref{state:general:pk>eps} of \cref{alg:general} effectively checks the condition \(p_k > \epsilon_{\rm p}\).
 	Log-like and inverse barriers do not experience this effect since \(\b*\leq 0\); see \cref{table:psi,thm:b*}.
 	Moreover, we do not observe this phenomenon in \cref{fig:lowrank} because, starting with the default values \(\alpha_0=\mu_0=1\), it is
 	\[
 	0
 	>
 	\frac{\b*(\nicefrac{\alpha_0}{\mu_0})}{\nicefrac{\alpha_0}{\mu_0}}
 	\leq
 	\frac{\b*(\nicefrac{\alpha_k}{\mu_k})}{\nicefrac{\alpha_k}{\mu_k}}
 	\nearrow
 	0
 	\]
 	for all barrier functions specified in \cref{table:psi}; see \cref{thm:b*/t*}.
 	More fundamentally, we can attribute the disparity in \cref{fig:lowrank_alpha} to the fact that, in contrast to the log-like and inverse barriers, the logarithmic barrier is only asymptotically well-behaved; see \cref{table:kappa}.

Overall, despite the fully nonconvex setting of problem \eqref{eq:matrixcompletion},
\pippo{} is able to solve these instances with only moderate values for the penalty parameter \(\alpha_k\).
Although these observations indicate that the assumptions behind \cref{thm:cvx:alphaConstant} could be relaxed,
the penalty exactness does not always take effect, as demonstrated in the following section.

\begin{figure}[tbh]
	\centering%
	\includetikz{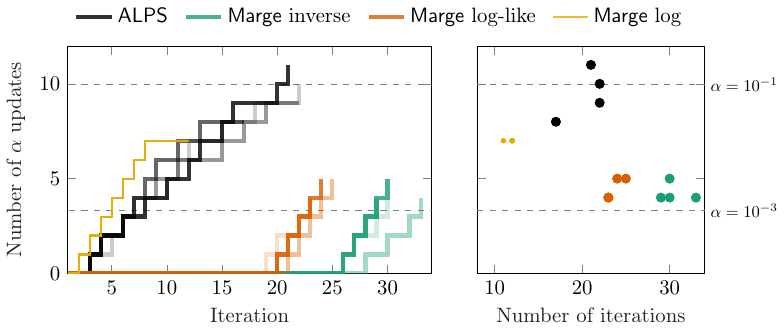}%
	\caption{%
		Matrix completion problem \eqref{eq:matrixcompletion}, smallest large instance:
		comparison of different solvers and variants in terms of updates for the penalty parameter \(\alpha_k\) along the iterations (left) and at the solution (right).
		All solvers start with the penalty value \(\alpha_0 = 10^{-4}\) for each of the 4 initializations.
		Fewer updates are expected to result in better-conditioned subproblems.
		Notice in the left panel that \protect\pippo{} log-like (orange) and inverse (green) increase \(\alpha_k\) only after several iterations and fewer times, thanks to the relaxed criterion at \cref{state:general:pk>eps}; \protect\pippo{} log (yellow) and \alps{} (black) update \(\alpha_k\) from the first iterations.%
	}%
	\label{fig:lowrank_alpha}
\end{figure}

		\subsection{Exact penalty behavior}\label{sec:Numerics:ExactPenalty}%
			After observing the bounded penalty behavior of \pippo{} in \cref{sec:Numerics:MatrixCompletion}, we present now an example problem where \pippo{} exhibits $\alpha_k\nearrow\infty$, hence it does \emph{not} boil down to an exact penalty method.
For this purpose it suffices to consider the two-dimensional \emph{convex} problem
\begin{equation}
	\label{eq:problem_lack_cq}
	\minimize_{\x\in\R^2}~
	\x_1 + \indicator_{\R_+}(\x_2)
	\quad
	\stt{}~
	\x_1^2 + \x_2 \leq 0
	,
\end{equation}
whose (unique) solution is the only feasible point \(\x^\star = (0,0)\).
Since there exists no suitable multiplier \(y^\star \), the minimizer \(\x^\star\) is not \KKT-optimal.
Hence, there is no contradiction with \cref{thm:cvx:alphaConstant}.

We intend to solve problem \eqref{eq:problem_lack_cq}
with tolerances \(\epsilon_{\rm p} = \epsilon_{\rm d} = 10^{-5}\),
initializing \pippo{} and \alps{} from 100 random points
generated according to $\x_i^0 \sim \normaldistrib(0,\sigma_x^2)$ with
large standard deviation $\sigma_x = 30$.

\paragraph*{Results}
	All solver variants find a primal-dual solution to \eqref{eq:problem_lack_cq}, up to the specified tolerances, for all starting points.
	The progression of \pippo{} and \alps{} in terms of penalty updates are summarized in \cref{fig:penalty_behavior}, in analogy with \cref{fig:lowrank_alpha}.
	The unbounded behavior of the penalty parameters $\seq{\alpha_k}$ appears evident for all solvers.
	Thus, $\alpha_k\nearrow\infty$ seems necessary to drive the constraint violation $p_k$ to zero,
	while the barrier parameter $\mu_k\searrow 0$ forces the complementarity slackness $s_k$.

\begin{figure}[tbh]
	\centering
		\includetikz{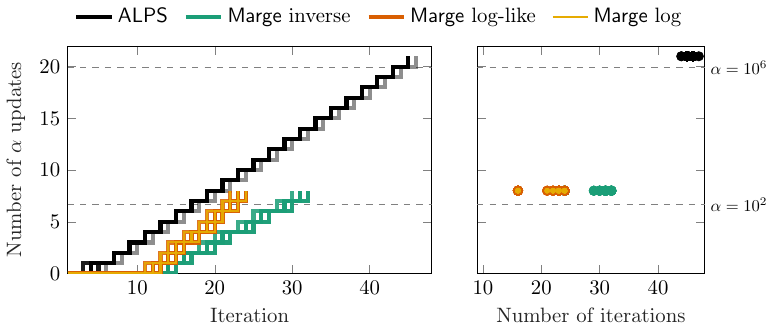}%
	\caption{%
		Convex problem without LICQ \eqref{eq:problem_lack_cq}, first 10 initializations:
		comparison of different solvers and variants in terms of updates for the penalty parameter \(\alpha_k\) along the iterations (left) and at the solution (right).
		All solvers start with the (default) penalty value \(\alpha_0=1\).
		The left panel depicts the solver trajectories for each initialization, indicating how many times the penalty parameter \(\alpha_k\) is updated during the solution process.
		Fewer updates are expected to result in better-conditioned subproblems.
		In all cases the sequence of penalty parameters \(\seq{\alpha_k}\) blows up,
		but variants of \protect\pippo{} terminate sooner and with less penalty updates than \alps{}.
		The results for logarithmic barriers almost overlap and always take less iterations than the inverse barrier, as expected.%
	}%
	\label{fig:penalty_behavior}
\end{figure}

Considering \pippo{}'s variants,
the log-like and log barriers are again almost indistinguishable and yield better results than the inverse one.
\alps{} performs poorly and always returns after more iterations and updates of the penalty parameter.
All runs of \pippo{} and \alps{} terminate with \(\alpha_k\) updated 8 and 21 times, respectively, namely increased up to $\alpha_0 \delta_\alpha^8 = 256$ and $\alpha_0 \delta_\alpha^{21} \approx 2.1\cdot 10^6$.
Then, it is clear that the performance of \alps{} is badly affected by the lack of regularity in \eqref{eq:problem_lack_cq}; without an appropriate dual estimate, \alps{} essentially falls back to a quadratic penalty scheme.
In contrast, all \pippo{}'s variants cope well with the lack of penalty exactness and operate consistently better than \alps{} in this scenario.

		\subsection{Handling equalities}\label{sec:Numerics:Equalities}%
			Even though equality constraints can be handled explicitly, it is important that \pippo{} can cope with hidden equalities too.
These may appear as the result of automatic model constructions, and are often difficult to identify by inspection.
Here we compare the behavior of \pippo{} when the problem specification has explicit equalities against the same problem but whose constraints are described using two inequalities each.
The latter approach not only increases the number of constraints, but it has also the drawback that the Mangasarian-Fromovitz constraint qualification (MFCQ) fails to hold at all feasible points.
Effectively, the redundancy introduced by splitting into two inequalities undermines the practical relevance of \cref{thm:KKT=>KKTa}; see \cite[\S 4.1.4]{curtis2012penalty}.

Consider quadratic programming (QP) problems of the form
\begin{equation}
	\label{eq:problem_equalities}
	\minimize_{\x\in\R^n}~
	\tfrac{1}{2} \x^\top \mat{Q} \x + \innprod{\vec{q}}{\x}
	\quad
	\stt{}~
	\mat{A} \x = \vec{b}
	,~
	\x^{\rm low} \leq \x \leq \x^{\rm upp}
\end{equation}
with matrices $\mat{Q}\in\R^{n\times n}$, $\mat{A}\in\R^{m\times n}$ and vectors $\vec{q},\x^{\rm low},\x^{\rm upp}\in\R^n$, $\vec{b}\in\R^m$ as problem data.
Problem \eqref{eq:problem_equalities} can be cast as \eqref{eq:P} with cost functions
\(f(\x)\coloneqq\tfrac{1}{2} \x^\top \mat{Q} \x + \innprod{\vec{q}}{\x}\) and
\(g(\x)\coloneqq\indicator_{[\x^{\rm low},\x^{\rm upp}]}(\x)\), and constraint function
\(\ceq(\x)\coloneqq \mat{A} \x - \vec{b}\).
We are interested in comparing the performance of \pippo{} (in different variants) with the two problem formulations described in \cref{sec:IP} to deal with equalities:
either by splitting into two inequalities (leading to the sum \(\psi_{\rho*}^{\pm}(t)\coloneqq\psi_{\rho*}(t)+\psi_{\rho*}(-t)\) as in \eqref{eq:psipm})
or by performing a combined marginalization (resulting in \(\psi_{\rho*}^{\rm eq}\)).
Hence, for each solver's variant and problem instance, we contrast these two formulations, symbolized by \pippo{}$^\pm$ and \pippo{}$^{\rm eq}$, respectively.

\paragraph*{Setup}
	Problem instances are generated as follows:
	we let either $\mat{Q} = \mat{M}\mat{M}^\top$ or $\mat{Q} = \mat{M}+\mat{M}^\top$, where the elements of $\mat{M}\in\R^{n\times n}$ are normally distributed, $\mat{M}_{i,j}\sim \normaldistrib(0,1)$, with only 10\% being nonzero.
	The linear part of the cost $\vec{q}$ is also normally distributed, i.e., $\vec{q}_i \sim \normaldistrib(0, 1)$.
	Simple bounds are generated according to a uniform distribution, i.e., $\x_i^{\rm low} \sim -\uniformdistrib(0, 1)$ and $\x_i^{\rm upp} \sim \uniformdistrib(0, 1)$.
	We set the elements of $\mat{A}\in\R^{m\times n}$ as $\mat{A}_{i,j}\sim \normaldistrib(0,1)$ with only 10\% being nonzero.
	To ensure that the problem is feasible,
	we draw an element $\widehat{\x} \in [\x^{\rm low},\x^{\rm upp}]$ (as $\widehat{x}_i = \x_i^{\rm low} + (\x_i^{\rm upp} - \x_i^{\rm low}) \vec{a}_i$ with $\vec{a}_i \sim \uniformdistrib(0,1)$) and set $\vec{b} = \mat{A} \widehat{\x}$.
	An initial guess is randomly generated for each problem instance, as $\x_i^0 \sim \normaldistrib(0,1)$, and shared across all solvers and formulations.

We consider problems with $m\in\{1,2,\ldots,20\}$ and $n = 10 m$, set the tolerances $\epsilon_{\rm p}=\epsilon_{\rm d}=10^{-5}$, and construct 10 instances for each size, for a total of 200 calls to each solver for each formulation.

\paragraph*{Results}
	Numerical results are visualized
	by means of pairwise (extended) performance profiles.
	Let $t_{s,p}^{\pm}$ and $t_{s,p}^{\rm eq}$ denote the evaluation metric of solver $s \in S$ on a certain instance $p \in P$ with the two formulations.
	Then, for each solver $s$, the corresponding pairwise performance profile displays the cumulative distribution
	$\varrho_s : [0,\infty) \to [0,1]$ of its performance ratio $\tau_{s,p}$, namely
	\begin{equation*}
		\varrho_s(\tau)
		\coloneqq
		\frac{|\set{p\in P}[\tau_{s,p}\leq\tau]|}{|P|}
		\quad\text{where}\quad
		\tau_{s,p}
		\coloneqq
		\frac{t_{s,p}^{\rm eq}}{t_{s,p}^{\pm}}
		.
	\end{equation*}
	Thus, the profile for solver $s$ indicates the fraction of problems $\varrho_s(\tau)$ for which solver $s$ invoked by \pippo{}$^{\rm eq}$ requires at most $\tau$ times the computational effort needed by the same solver $s$ when invoked by \pippo{}$^{\pm}$.

As depicted in \cref{fig:bilateral}, all pairwise performance profiles cross the unit ratio with at least 70\% problems solved, meaning that, across all variants, \pippo{}$^{\rm eq}$ is a more effective formulation than \pippo{}$^\pm$ for a large majority of problems.
Thus, all solver variants benefit from the tailored handling of equality constraints, especially in terms of gradient evaluations (whose profiles are all above 85\% at the unit ratio).
The flat region of \(\psi_{\rho*}^{\pm}\) displayed in \cref{fig:b_rho_equality,fig:b_mu_equality} is arguably responsible for hindering the performance of \pippo{} when equalities are split into two inequalities.
In the possibly nonconvex case (bottom panels of \cref{fig:bilateral}), the narrower advantage of \pippo{}$^{\rm eq}$ could stem from the fact that solvers might end up in different local solutions, or even spurious ones due to the reformulation with two inequalities.
Moreover, the smaller benefit of \pippo{}$^{\rm eq}$ observed in terms of runtime over gradient evaluations could be ascribed to the slightly more complicated computation of \(\psi_{\rho*}^{\rm eq}\) over \(\psi_{\rho*}\).
Regardless, all variants exhibited a robust performance with both formulations, confirming that our algorithmic framework can endure redundant, degenerate constraints and hidden equalities.

\begin{figure}[tbh]
	\centering%
	\includetikz[width=\linewidth]{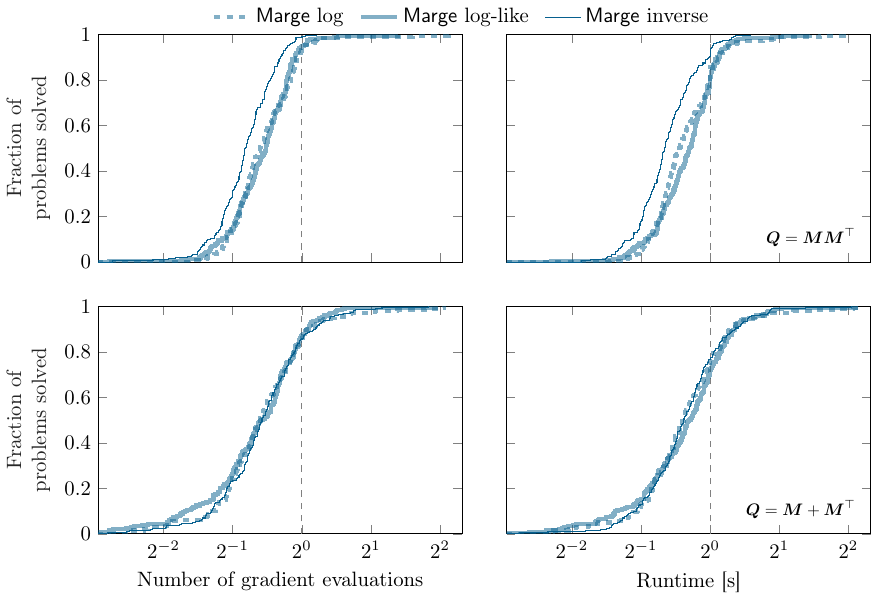}%
	\caption{%
		Quadratic programs \eqref{eq:problem_equalities}: comparison of different solvers and formulations using pairwise performance profiles, relative to gradient evaluations (left) and runtime (right), for \protect\pippo$^{\rm eq}$ (explicit equality) over \protect\pippo$^{\pm}$ (split into two inequalities).
		Profiles located in the top-left indicate that \protect\pippo{}$^{\rm eq}$ tends to outperform \protect\pippo$^{\pm}$.
		Across all barrier functions, \protect\pippo{}$^{\rm eq}$ is more efficient than \protect\pippo{}$^\pm$ for both convex (top panels) and possibly nonconvex problems (bottom panels).%
	}%
	\label{fig:bilateral}
\end{figure}

	\section{Final remarks}
		We proposed \pippo{}, an optimization framework for the numerical solution of constrained structured problems in the fully nonconvex setting.
\pippo{} combines (exact) penalty and barrier approaches through a marginalization step, which not only preserves the problem size by avoiding auxiliary variables, but also enables the adoption of generic subsolvers.
In particular, by extending the domain of the subproblems' smooth objective term, the proposed methodology overcomes the need for safeguards within the subsolver and the difficulty of accelerating it, a major drawback of \ipprox{} \cite{demarchi2024interior}.
Under mild assumptions, our theoretical analysis established convergence results on par with those typical for nonconvex continuous optimization.
Most notably, all feasible accumulation points are asymptotically KKT optimal.
We validated and compared our approach numerically with problems arising in data science, studying scalability and the effect of accuracy requirements.
Furthermore, illustrative examples confirmed the robust behavior of \pippo{} on badly formulated problems and degenerate cases.

The methodology in this paper could be applied to a combination of barrier and augmented Lagrangian approaches.
By generating a smoother penalty-barrier term, this strategy could benefit from the more effective performance of subsolvers.
However, this development comes with the additional challenge of designing suitable updates for the Lagrange multipliers.
Future research may also focus on specializing the proposed framework to classical nonlinear programming, taking advantage of the special structure and linear algebra.
Finally, mechanisms for rapid infeasibility detection and guaranteed existence of subproblems' solutions should be investigated.

	\appendix
		\section{Auxiliary results and missing proofs}
			This appendix contains some auxiliary results and proofs of statements referred to in the main body.

\begin{lemma}[Properties of the barrier \(\b\)]%
	Any function \(\b\) as in \cref{ass:b} satisfies the following:%
	\begin{enumerate}[widest*=4,align=left]
	\item \label{thm:brange}%
		\(\lim_{t\to-\infty}\b(t)=\inf\b\) and \(\lim_{t\to0^-}\b(t)=\lim_{t\to0^-}\b'(t)=\infty\).
	\item \label{thm:b*}%
		The conjugate \(\b*\) is continuously differentiable on the interior of its domain \(\interior\dom\b*=\R_{++}\) with \(\b*'<0\), and satisfies \(\b*(0)=-\inf\b \) and \(\lim_{\conj t\to\infty}\b*(\conj t)=-\infty\).
	\item \label{thm:bb*}%
		\(\b*(t^*)=\b*'(t^*)t^*-\b(\b*'(t^*))\) for any \(t^*>0\).
	\item \label{thm:b*/t*}%
		The function \((0,\infty)\ni\conj t\mapsto\nicefrac{\b*(\conj t)}{\conj t}=t-\nicefrac{\b(t)}{\b'(t)}\), where \(t\coloneqq\b*'(\conj t)\), is strictly increasing for \(\conj t\) large enough with \(\lim_{\conj t\to\infty}\nicefrac{\b*(\conj t)}{\conj t}=0\).
			If \(\inf\b\geq0\), then it is strictly increasing on the whole \((0,\infty)\).
			\end{enumerate}
\end{lemma}
\begin{proof}~
	\begin{proofitemize}
	\item \ref{thm:brange}~
		Trivial because of strict monotonicity on \((-\infty,0)\) (since \(\b'>0\)).

	\item \ref{thm:b*}~
		It follows from the definition of Fenchel conjugate that \(\b*(0)=-\inf\b\), see also \cite[Prop. 13.10(i)]{bauschke2017convex}.
			Notice that \(\b\) is (essentially) strictly convex and \emph{essentially smooth}, in the sense that \(\b'(t)\to\infty\) as \(t\to0^-\), 0 being the only point in the boundary of \(\dom\b\).
			As such, the conjugate \(\b*\) enjoys the same properties, with \(\interior\dom\b*=\range\b'=\R_{++}\) by virtue of \cite[Thm. 26.1 and 26.3]{rockafellar1970convex}.
				For the same reason, \(\range\b*'=\interior\dom\b=\R_{--}\), hence \({\b*}'<0\) on \((0,\infty)\).
		Finally, since \(\inf\b*=-\b(0)=-\infty\), we conclude that \(\lim_{t^*\to\infty}\b*(t^*)=-\infty\).

	\item \ref{thm:bb*}~
		This is a standard result of Fenchel conjugacy, see e.g. \cite[Prop. 16.10]{bauschke2017convex}, here specialized to the fact that \(\range\b'=\R_{++}\).

	\item \ref{thm:b*/t*}~
		Observe that
		\(
			\left(\nicefrac{\b*(t^*)}{t^*}\right)'
		=
			\nicefrac{\b(t)}{(t^*)^2}
		\)
		for \(t^*>0\) and \(t\coloneqq\b*'(t^*)\to0^-\) as \(t^*\to\infty\).
		Since \(\b(0^-)=\infty\), for \(t^*\) large enough (or for any \(t^*>0\) if \(\inf\b\geq0\)) this derivative is strictly positive, and as such the function strictly increasing.
		Lastly,
		\begin{equation}\label{eq:b*/t*}
			\lim_{t^*\to\infty}\tfrac{\b*(t^*)}{t^*}
		=
			\lim_{t^*\to\infty}{\b*}'(t^*)
		=
			\lim_{\b'(t)\to\infty}t
		=
			\lim_{t\to0^-}t
		=
			0,
		\end{equation}
		where the first equality uses L'H\^opital's rule.
	\qedhere
	\end{proofitemize}
\end{proof}

\begin{appendixproof}[Lemma ]{thm:KKTa}%
	The Lagrangian associated to \eqref{eq:Qa} reads
	\begin{align*}
		\mathcal L(\x,\s,\sseq;\lam^+,\lam^-,\y)
	={} &
		\cost(\x)
		+
		\alpha\innprod{\ones}{\s}
		+
		\indicator_{\R_+^m}(\s)
		+
		\alpha\innprod{\ones}{\sseq}
	\\
	&
		+
		\innprod{\y}{\c(\x)-\s}
		+
		\innprod{\lam^+}{\ceq(\x)-\sseq}
		-
		\innprod{\lam^-}{\ceq(\x)+\sseq},
	\end{align*}
	so that the corresponding KKT conditions are
	\[
		\begin{cases}
			\zeros
			\in
			\partial\cost(\x)+\trans{\jac\c(\x)}\y+\trans{\jac\ceq(\x)}(\lam^+-\lam^-)
		\\
			0
			\in
			\alpha-\y_i+\ncone_{\R_+}(\s_i)
		\\
			0
			=
			\alpha-\lam^+_j-\lam^-_j
		\end{cases}
	\quad
		\begin{cases}
			\c_i(\x)\leq\s_i
		\\
			|\ceq_j(\x)|\leq\sseq_j
		\\
			\y_i,\lam_j^\pm\geq0
		\end{cases}
	\quad
		\begin{cases}
			0=\y_i(\c_i(\x)-\s_i)
		\\
			0=\lam^+_j(\sseq_j-\ceq_j(\x))
		\\
			0=\lam^-_j(\sseq_j+\ceq_j(\x))
		\end{cases}
	\]
	where \(i=1,\dots,m\) and \(j=1,\dots,\meq\).
	Here, the first set of conditions corresponds to Lagrangian stationarity (LS), the second one to primal and dual feasibility (PDF), and the last one to complementarity slackness (CS).

	Suppose that \(\sseq_j>|\ceq_j(\x)|\); then, CS implies that \(\lam_j^\pm=0\), contradicting the fact that \(\lam_j^++\lam_j^-=\alpha\) in LS.
	Thus, \(\sseq=|\ceq(\x)|\) must hold.
	Suppose instead that \(\s_i>\c_i(\x)\); then, \(\y_i=0\) by CS, and the second condition in LS then implies that \(\s_i=0\) (for otherwise \(\ncone_{\R_+}(\s_i)=\set{0}\)).
	Either way, since \(y_i-\alpha\in\ncone_{\R_+}(\s_i)\subseteq\R_-\), one has that \(y_i-\alpha\leq0\).
	These observations show that \(\s=[\c(\x)]_+\) and that \(\zeros\leq\y\leq\alpha\ones\).

	Set \(\yeq\coloneqq\lam^+-\lam^-\), which combined with the last condition in LS yields that
	\(
		\lam^+=\frac{1}{2}(\alpha\ones+\yeq)
	\)
	and
	\(
		\lam^-=\frac{1}{2}(\alpha\ones-\yeq)
	\).
	Since \(\lam^\pm\geq\zeros\) by PDF, one has that \(|\yeq|\leq\alpha\ones\).
	With these substitutions, observing that
	\begin{equation}\label{eq:c+ceq+-}
		\sseq-\ceq(\x)=2[\ceq(\x)]_-,
	\quad
		\sseq+\ceq(\x)=2[\ceq(\x)]_+,
	\quad\text{and}\quad
		\s-\c(\x)=[\c(\x)]_-,
	\end{equation}
	the KKT conditions simplify as
	\[
		\begin{cases}
			\zeros
			\in
			\partial\cost(\x)+\trans{\jac\c(\x)}\y+\trans{\jac\ceq(\x)}\yeq
		\\
			\zeros\leq\y\leq\alpha\ones
		\\
			|\yeq|\leq\alpha\ones
		\end{cases}
	\quad
		\begin{cases}
			0=\y_i[\c_i(\x)]_-,~
			y_i-\alpha\in \ncone_{\R_+}([\c_i(\x)]_+)
		\\
			0=(\alpha+\yeq_j)[\ceq_j(\x)]_-
		\\
			0=(\alpha-\yeq_j)[\ceq_j(\x)]_+
		\end{cases}
	\]
	where \(i=1,\dots,m\) and \(j=1,\dots,\meq\).
	Noticing that \(\ncone_{\R_+}([\c_i(\x)]_+)=\set{0}\) when \([\c_i(\x)]_+>0\), \eqref{KKTa} are obtained.
	Conversely, by reverting \(\yeq=\lam^+-\lam^-\) and using \eqref{eq:c+ceq+-} to substitute \([\ceq(\x)]_\pm\) and \([\c(\x)]_+\) one reobtains the KKT conditions for problem \eqref{eq:Qa}.
\end{appendixproof}

\begin{appendixproof}[Theorem ]{thm:psi}%
	We start by observing that
	\[
		\psi_{\rho*}(t)
	\defeq
		\min_{z\geq0}\set{\rho*z+\b(t-z)}
	=
		\min_{z\in\R}\set{\rho*|z|+\b(t-z)},
	\]
	owing to the fact that \(\b\) is increasing and consequently \(\inf_{z\leq 0}\b(t-z)=\b(t)\) for any \(t\).
		Hence, \(\psi_{\rho*}\) is the \emph{\(\rho*\)-Pasch-Hausdorff envelope} of the convex function \(\b\) as in \cite[Def. 12.16]{bauschke2017convex}, and is itself convex by virtue of \cite[Prop. 12.11]{bauschke2017convex}.
		Since \(\b''>0\), and \(\b'(0^-)=\infty\), there exists \(\rho<0\) such that \(\b'(\rho)=\rho*\).
		In particular, \(\b'((-\infty,\rho])=(0,\rho*]\) and \(\b'((\rho*,0))=(\rho^*,\infty)\), implying that \(\b\) is \(\rho*\)-Lipschitz continuous on \((-\infty,\rho]\).
		The expression \eqref{eq:psieq} and \(\rho*\)-Lipschitz continuity then follow from \cite[Prop. 12.17(i)]{bauschke2017convex}, and in turn so does the expression \eqref{eq:psi'} of the derivative, which is clearly globally Lipschitz continuous as well.
		Finally, that \(\psi_{\rho*}\circ c\) is convex whenever \(c\) is convex follows from the fact that \(\psi_{\rho*}\) is increasing (additionally to being convex).
\end{appendixproof}

\begin{appendixproof}[Theorem ]{thm:psieq}%
	Function \(F(z,t)\coloneqq \rho*z+\b(t-z)+\b(-t-z)\) being minimized on the right-hand side of \eqref{eq:psieqdef} is convex, hence so is its marginalized (wrt \(z\)) function \(\psi_{\rho*}^{\rm eq}\).
	For every \(t\in\R\), \(z\mapsto F(z,t)\) is proper, lsc, strictly convex, coercive, and differentiable on its (open) domain, and thus admits a unique minimizer \(z_{\rho*}(t)\), this being the (unique) zero of the derivative, that is, such that \eqref{eq:psieq:s*} holds.
	In particular, \(\psi_{\rho*}^{\rm eq}\) is convex and finite valued, and thus everywhere subdifferentiable.
	Appealing to \cite[Thm. 10.13]{rockafellar1998variational} and denoting \(z=z_{\rho*}(t)\), its (regular, or equivalently, convex) subdifferential satisfies
	\[\textstyle
		\emptyset
	\neq
		\hat\partial\psi_{\rho*}^{\rm eq}(t)
	\subseteq
		\set{y}[\binom{0}{y}\in\hat\partial F(z,t)]
	=
		\set{y}[\binom{0}{y}\in\binom{\rho*-\b'(t-z)-\b'(-t-z)}{\b'(t-z)-\b'(-t-z)}]
	\subseteq
		\set{\b'(t-z)-\b'(-t-z)}.
	\]
	This shows that \(\psi_{\rho*}^{\rm eq}\) is everywhere differentiable with derivative
	\[
		(\psi_{\rho*}^{\rm eq})'(t)
	=
		\b'(t-z_{\rho*}(t))-\b'(-t-z_{\rho*}(t))
	=
		\rho*-2\b'(-t-z_{\rho*}(t)),
	\]
	as claimed, where the second identity follows from \eqref{eq:psieq:s*}.
	Notice that \eqref{eq:psieq:s*} also implies that \(\b'(\pm t-z_{\rho*}(t))\leq\rho*\) (by \(\b'>0\)); since \(\b'\) is increasing, one must have that \(\pm t-z_{\rho*}(t)\leq\b*'(\rho*)\eqqcolon\rho\), yielding the claimed bound \(z_{\rho*}(t)>|t|-\rho\).
	Notice further that \((\psi_{\rho*}^{\rm eq})'(t)<\rho*\), since \(\b'>0\), and consequently \((\psi_{\rho*}^{\rm eq})'(t)>-\rho*\) as well by symmetry.
	In particular, \(\psi_{\rho*}^{\rm eq}\) is globally \(\rho*\)-Lipschitz continuous.

	We next turn to Lipschitz differentiability.
	We first demonstrate that the mapping \(\R_+\ni t\mapsto z_{\rho*}(t)\) is nonexpansive.
	Fix \(t'>t\geq0\) and let \(\varepsilon\coloneqq z_{\rho*}(t')-z_{\rho*}(t)-(t'-t)\); since, apparently, \(z_{\rho*}(t')\geq z_{\rho*}(t)\), the claim is proven once we show that \(\varepsilon\leq0\).
	It follows from \eqref{eq:psieq:s*} that
	\begin{align*}
		\b'(t-z_{\rho*}(t))+\b'(-t-z_{\rho*}(t))
	={} &
		\b'(t'-z_{\rho*}(t'))+\b'(-t'-z_{\rho*}(t'))
	\\
	={} &
		\b'(t-z_{\rho*}(t)-\varepsilon)+\b'(t-z_{\rho*}(t)-2t'-\varepsilon).
	\end{align*}
	For the top left-hand side to equal the bottom right-hand side \(\varepsilon<0\) must hold, for otherwise \(\b'(t-z_{\rho*}(t)-\varepsilon)<\b'(t-z_{\rho*}(t))\) and \(\b'(t-z_{\rho*}(t)-2t'-\varepsilon)<b'(-t-z_{\rho*}(t))\) (since \(\b'\) is increasing and \(t'>0\)).
	This shows that \(|z_{\rho*}(t')-z_{\rho*}(t)|\leq|t'-t|\), as claimed.
	Next, observe that
	\[
		\bigl|
			(\psi_{\rho*}^{\rm eq})'(t')
			-
			(\psi_{\rho*}^{\rm eq})'(t)
		\bigr|
	=
		2\bigl|
			\b'(-t-z_{\rho*}(t))
			-
			\b'(-t'-z_{\rho*}(t'))
		\bigr|,
	\]
	and both \(-t-z_{\rho*}(t)\) and \(-t'-z_{\rho*}(t')\) are larger than \(\rho=\b*'(\rho*)<0\), as shown above.
	In particular,
	\[
		\bigl|
			(\psi_{\rho*}^{\rm eq})'(t')
			-
			(\psi_{\rho*}^{\rm eq})'(t)
		\bigr|
	\leq
		2B
		\bigl|
			t'-t
			+
			z_{\rho*}(t')-z_{\rho*}(t)
		\bigr|
	\leq
		4B|t'-t|,
	\]
	where \(B\coloneqq\sup_{(-\infty,\rho]}\b''<\infty\) is a finite quantity (by the properties of \(\b\) in \cref{ass:b}) that depends only on \(\rho*\).
	This shows that \((\psi_{\rho*}^{\rm eq})'\) is \(4B\)-Lipschitz continuous on \(\R_+\), hence on the entire \(\R\) by symmetry.

	Lastly, take \(t>0\) and observe that, since \(z_{\rho*}(t)>|t|-\rho>0\) and \(\b'\) is increasing, one has
	\(
		(\psi_{\rho*}^{\rm eq})'(t)
	=
		\rho*-2\b'(-t-z_{\rho*}(t))
	>
		\rho*-2\b'(-t)
	\).
	By symmetry, the claimed inequality \(|(\psi_{\rho*}^{\rm eq})'(t)|>\rho*-2\b'(-|t|)\) follows.
\end{appendixproof}

\begin{appendixproof}[Theorem ]{thm:rho-infty}%
		We start by observing that when \(\b>0\) the pointwise monotonic decrease of \(\psi_{\rho*}/\rho*\) and \(\psi_{\rho*}^{\rm eq}/\rho*\) as \(\rho*\) grows is apparent from the respective definitions \eqref{eq:psi} and \eqref{eq:psieq}.
		Next,
		the claimed limit of \(\psi_{\rho*}\) follows from the expression \eqref{eq:psi} and \cref{thm:b*/t*}.
	As to \(\psi_{\rho*}^{\rm eq}\),
	notice that it satisfies
	\[
		\psi_{\rho*}^{\rm eq}\bigl(t\bigr)
	\leq
		\inf_{z \geq 0}\set{
			\rho*z
			+
			2\b\bigl(|t|-z\bigr)
		}
	=
		2\psi_{\nicefrac{\rho*}{2}}\bigl(|t|\bigr),
	\]
	where the inequality owes to the fact that \(\b'>0\) (hence that \(\b\) is increasing).
		Next, fix \(t\in\R\) and plug \(z_{\rho*}(t)=z_{\rho*}(|t|)\) into \eqref{eq:psidef} to obtain a bound
		\[
			\psi_{\rho*}(|t|)
		\defeq
			\min_{z\geq0}\set{\rho*z+\b(|t|-z)}
		\leq
			\rho*z_{\rho*}(|t|)+\b(|t|-z_{\rho*}(|t|))
		=
			\psi_{\rho*}^{\rm eq}(t)
			-
			\b(-|t|-z_{\rho*}(t)).
		\]
		Observe that \(z_{\rho*}(t)\searrow |t|\) as \(\rho*\to\infty\), implying that the term \(\b(-|t|-z_{\rho*}(t))\) is positive for \(\rho*\) large enough (or for any \(\rho*\) in case \(\b>0\)).
		Thus, for any fixed \(t\) and \(\rho*\) large enough one has that
		\(
			\psi_{\rho*}^{\rm eq}(t)
		\geq
			\psi_{\rho*}(|t|)
		\),
		which combined with the previous bound results in
		\[
			\psi_{\rho*}\bigl(|t|\bigr)
		\leq
			\psi_{\rho*}^{\rm eq}\bigl(t\bigr)
		\leq
			2\psi_{\nicefrac{\rho*}{2}}\bigl(|t|\bigr)
		\]
		for all \(\rho*\) large enough (depending on \(t\)).
		Dividing by \(\rho*\) and letting \(\rho*\to\infty\), it follows from the earlier claim on \(\psi_{\rho*}\) that the lower and the upper bounds both converge to
	\(
		\bigl[|t|\bigr]_+=|t|
	\),
	demonstrating the claim.
\end{appendixproof}

		\section{Barrier properties}\label{sec:Numerics:BarrierProp}%
			The analysis in \cref{sec:barrierProp} of the barrier's properties is corroborated by examining the performance of solver variants in terms of outer iterations.
Computational results relative to the nonnegative PCA example of \cref{sec:Numerics:NonnegPCA} are depicted in \cref{fig:nonneg_pca_iter}, where (i) the log-like and log barriers have almost indistinguishable profiles and (ii) they invariably demand the solution of fewer subproblems than the inverse barrier.
Although observation (i) is supported also by the results in \cref{fig:penalty_behavior},
it does not hold in general, as pointed out with the experiments summarized in \cref{fig:lowrank_alpha} and the associated discussion.
Regardless, both observations are in agreement with the analysis in \cref{sec:barrierProp}, in particular with the log-like barrier having a better (in fact, optimal) behavior profile than the inverse barrier.
Profiles with an analogous pattern were observed also for the matrix completion task of \cref{sec:Numerics:MatrixCompletion}, but with a less pronounced discrepancy between the inverse and other barrier functions.
Interestingly, \cref{fig:nonneg_pca_iter} suggests that this assessment remains valid also for \ipprox{}'s profiles, although not being covered by our theory.

\begin{figure}[htb]
	\includetikz[width=\linewidth]{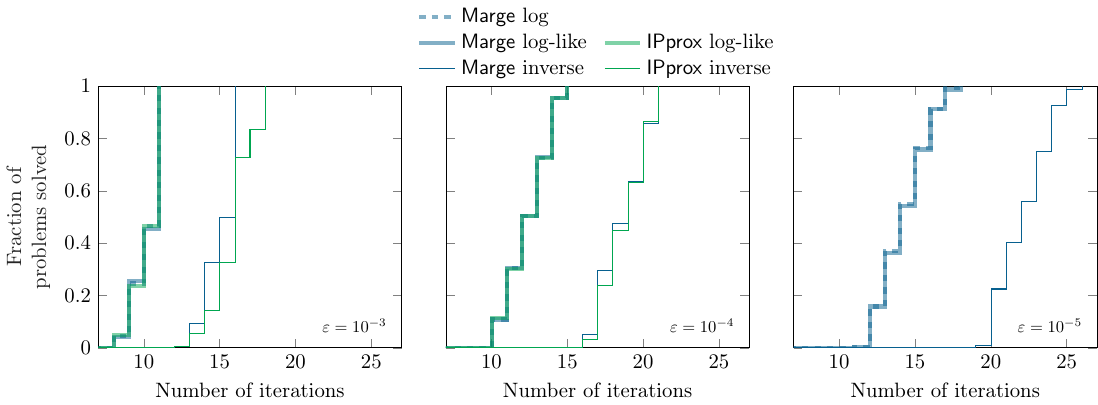}%
	\caption[]{%
		Nonnegative PCA problem \eqref{eq:nonneg_pca}, small and large instances: comparison of solvers with low, medium and high accuracy (left to right) using data profiles relative to the number of outer iterations.
		With high accuracy, results for \ipprox{} (green solid lines) are not included due to excessive runtime; in this case the other solvers are initialized with a possibly infeasible guess.
		The comparisons in terms of \emph{outer} iterations for \pippo{} (blue) and \ipprox{} (green) confirms the theoretical appeal of ``well-behaved'' logarithmic (thick dashed lines) and log-like (thick solid lines) barriers for the former method, with the two profiles almost perfectly overlapping.%
	}%
	\label{fig:nonneg_pca_iter}%
\end{figure}

	\phantomsection
	\addcontentsline{toc}{section}{References}%
	\bibliographystyle{jnsao}
	\bibliography{TeX/references.bib}

\end{document}